\documentclass[11pt]{amsart}

\usepackage{amsmath, amsthm, amssymb, mathtools, mathpazo, euler, microtype}
\usepackage{graphicx}
\usepackage{tikz}
\usepackage[lined,algonl,boxed,norelsize]{algorithm2e}
\usepackage[noend]{algpseudocode} 
\usepackage{pdflscape}
\usepackage{comment}

\usepackage[margin=1 in]{geometry}
\AtBeginDocument{%
   \def\MR#1{}
}

\newcommand{\Z}{\mathbb{Z}}
\newcommand{\N}{\mathbb{N}}

\newcommand{\Q}{\mathbb{Q}}

\newcommand{\E}{\mathbb{E}}
\newcommand{\kk}{\mathbb{K}}
\newcommand{\pdim}{\operatorname{pdim}}
\newcommand{\reg}{\operatorname{reg}}
\newcommand{\depth}{\operatorname{depth}}
\newcommand{\Tor}{\operatorname{Tor}}
\newcommand{\lk}{\operatorname{lk}}
\newcommand{\st}{\operatorname{st}}
\newtheorem{theorem}{Theorem}
\newtheorem{lemma}[theorem]{Lemma}

\newtheorem{proposition}[theorem]{Proposition}
\newtheorem{corollary}[theorem]{Corollary}

\newtheorem{remark}[theorem]{Remark}
\theoremstyle{definition}
\newtheorem{conjecture}[theorem]{Conjecture}

\newtheorem{example}[theorem]{Example}

\newcommand\defi[1]{\textit{\color{blue}#1}}

\definecolor{darkblue}{rgb}{0,0,0.7}
\definecolor{darkred}{rgb}{0.7,0,0}
\usepackage[colorlinks,linkcolor=darkred, citecolor=darkblue,
pagebackref=true,pdftex]{hyperref}

\title{Random subcomplexes and Betti numbers of random edge ideals}

\author{Anton Dochtermann}
\address{Texas State University}
\email{dochtermann@txstate.edu}
\urladdr{\url{https://dochtermann.wp.txstate.edu/}}

\author{Andrew Newman}
\address{Carnegie Mellon University}
\email{anewman@andrew.cmu.edu}
\urladdr{\url{https://sites.google.com/view/andrewnewman775/home/}}

\date{\today}

\begin{document}

\begin{abstract}
We study homological properties of random quadratic monomial ideals in a polynomial ring $R = \kk[x_1, \dots x_n]$, utilizing methods from the Erd\H{o}s--R\'enyi model of random graphs. Here for a graph $G \sim G(n, p)$ we consider the coedge ideal $I_G$ generated by monomials corresponding to the missing edges of $G$, and study Betti numbers of $R/I_G$ as $n$ tends to infinity. Our main results involve setting the edge probability $p = p(n)$ so that asymptotically almost surely the Krull dimension of $R/I_G$ is fixed. Under these conditions we establish various properties regarding the Betti table of $R/I_G$, including sharp bounds on regularity and projective dimension and distribution of nonzero normalized Betti numbers. These results extend work of Erman and Yang who studied such ideals in the context of conjectured phenomena in the nonvanishing of asymptotic syzygies.  Along the way we establish results regarding subcomplexes of random clique complexes as well as notions of higher-dimensional vertex $k$-connectivity that may be of independent interest.
\end{abstract}
\maketitle

\section{Introduction}\label{sec:intro}
Suppose $\kk$ is a field and let $R = \kk[x_1, x_2, \dots, x_n]$ denote its polynomial ring in $n$ variables. If $M$ is a graded $R$-module, a minimal free resolution of $M$ gives rise to several homological invariants reflected in its \emph{Betti table}, where the entry in row $k$ and column $i$ is given by the Betti number $\beta_{i,i+k}(M)$. In recent years there has been interest in understanding the ``typical'' shape of these Betti tables as $n$ becomes large, especially for the case of a monomial ideal $I$ and its quotient ring $R/I$. Questions regarding such ``asymptotic syzygies'' are outside the range computable by current software, and hence require new techniques.

In this paper we study Betti tables of random quadratic squarefree monomial ideals via the Erd\H{o}s--R\'enyi model of random graphs.  We let $G(n,p)$ denote the Erd\H{o}s--R\'enyi random graph on $n$ vertices with edge probability $p = p(n)$ and construct the \emph{coedge ideal} $I_G$, the squarefree quadratic monomial ideal whose generators correspond to all missing edges of $G$. The ideal $I_G$ can also be recovered as the \emph{Stanley--Reisner ideal} of $\Delta(G)$, where $\Delta(G)$ is the \emph{clique complex} of $G$.  Hochster's formula provides a way to compute the Betti numbers of $I_G$ (and hence the quotient ring $R/I_G$) in terms of the topology of subcomplexes of $\Delta(G)$. Our approach will be to utilize and extend methods from the study of random clique complexes, first introduced in \cite{Kah2014}, to study the Betti table of $R/I_G$. In this context we use $\Delta(n,p)$ to denote the random clique complex $\Delta(G(n,p))$.

In \cite{ErmYan} Erman and Yang take a similar approach, motivated by conjectures in algebraic geometry regarding the behavior of asymptotic syzygies of increasingly positive embeddings of a variety (see for instance \cite{EinErmLaz}).  They adopt the same Erd\H{o}s--R\'enyi model of a random coedge ideal to produce a family $\{I_n\}$ of random ideals where $\pdim(I_n) \rightarrow \infty$.  In  \cite[Theorem 1.3]{ErmYan} Erman and Yang prove that if $G \sim G(n,p)$ with $\frac{1}{n^{1/d}} \ll p \ll 1$, then for all $k \leq d+1$ the density of nonzero entries in the $k$th row of the Betti table of $R/I_G$ will approach one in probability. We refer to Section \ref{sec:log} for more details.

Erman and Yang also establish some results on the individual Betti numbers of random coedge ideals in certain regimes for $p$. For instance if we take $p=n^{- \alpha}$ then for any \emph{fixed} $i$ and $k$ we have $\beta_{i, i+k}(R/I_G) = 0$ with high probability if $\alpha > 1/(k- 1)$, whereas $\beta_{i, i+k}(R/I_G) \neq 0$ if $1/(k - 1) > \alpha > 0$ \cite[Theorem 1.6]{ErmYan}. The authors use this result to provide a lower bound for the regularity of coedge ideals for these choices of $p$.  For $p = n^{-\alpha}$ with $1/(d + 1) < \alpha < 1/d$ these results provide an asymptotic description of what the Betti table of a random coedge ideal looks like in all rows between $1$ and $d+1$ and in all columns between $1$ and any constant.  The question of what happens for rows larger than $d+1$, as well as what the Betti table looks like ``on the right side'' is partly what motivates our work.

The methods employed in \cite{ErmYan} use Hochster's formula to interpret the Betti numbers of coedge ideals in terms of the topology of induced subcomplexes of the clique complex of the underlying graph. A main strategy of \cite{ErmYan} is to analyze the appearance of high-dimensional cross-polytopes as induced subcomplexes. As these complexes carry nonzero integral homology, most of the results in \cite{ErmYan} are independent of the underlying coefficient field ${\mathbb K}$.  This approach has also been fruitful in the study of clique complexes of random graphs (often referred to as \emph{random clique complexes}), which is itself an emerging field. We refer to \cite{Kahsurvey} for a good survey or Section \ref{sec:background} for more details relevant to our study.


\subsection{Our results}
In this paper we further utilize and extend methods from the study of random clique complexes to obtain new results regarding algebraic properties of random quadratic monomial ideals. We employ a range of techniques including notions of collapsibility and Garland's method regarding the spectral gap of graph Laplacians. In this way we answer some questions from \cite{ErmYan} as well as provide new ways to understand the underlying algebraic objects.   We emphasize that some of our results will depend on the choice of the coefficient field ${\mathbb K}$ of the underlying polynomial ring $R$.

As is typical in the study of random objects, our results are asymptotic. The Erd\H{o}s--R\'{e}nyi random graph $G(n, p)$ is the probability space of graphs on $n$ vertices sampled by including each edge independently with probability $p$. Here $p = p(n)$ is often a function of $n$ that tends to 0 at some rate as $n$ tends to infinity. For a graph property $\mathcal{P}$ one says that $G \sim G(n, p)$ has property $\mathcal{P}$ \defi{with high probability}, abbreviated w.h.p, or \defi{asymtotically almost surely}, abbreviated a.a.s., provided the probability that $G$ has property $\mathcal{P}$ tends to 1 as $n$ tends to infinity. By passing to the coedge ideal of a random graph $G$ we can also study asymptotic behavior of $R/I_G$ with respect to certain algebraic properties. 

Our primary approach here involves setting the edge probability of our underlying random graph $G(n,p)$ to be $p(n) = n^{-\alpha}$ for some $\alpha \in (0,\infty)$. In this sparse regime standard results from random graph theory imply that the dimension of the underlying clique complex is bounded in terms of $\alpha$ with high probability. In particular we have that if $\alpha \in (2/(k + 1), 2/k)$ then $G(n,p)$ will have clique number $k+1$ with probability tending to one, see Section \ref{sec:background} for more details.  It is known that the Krull dimension of $R/I_G$ is given by one plus the dimension of the clique complex of $G$, and hence is equal to the clique number of $G$.   Therefore fixing $p$ in this range provides a model for a random quadratic monomial ideal where $R/I_G$ has fixed Krull dimension. Our results describe typical algebraic properties of such rings.

\begin{example}\label{ex:running}
As an example of our construction we take $n = 20$ and $\alpha = 0.27$ and a single instance of $G \sim G(n, n^{-\alpha})$.  We see the Betti table of $R/I_G$ in Table \ref{sampleBettiTable}, as computed in GAP \cite{GAP4} using the \texttt{simpcomp} package \cite{simpcomp} for computing simplicial homology.
\end{example}

\begin{table}[h]
\begin{tiny}
\begin{tabular}{cccccccccccccccccc}
0 & 1 & 2 & 3 & 4 & 5 & 6 & 7 & 8 & 9 & 10 & 11 & 12 & 13 & 14 & 15 & 16 & 17   \\ \hline
1 & 0 & 0 & 0 & 0 & 0 & 0 & 0 & 0 & 0 & 0 & 0 & 0 & 0 & 0 & 0 & 0 & 0 \\
- & 106 & 867 & 3506 & 8852 & 15496 & 20257 & 20942 & 17682 & 12213 & 6764 & 2914 & 938 & 212 &  30 &  2 & 0 & 0   \\
 - & - & 175 & 2472 & 15558 & 58439 & 148003 & 270285 & 370495 & 390104 & 318911 & 202556 &  99092 & 36628 & 9890 & 1839 & 210 & 11    \\
 - & - & - & 2 & 30 & 207 & 871 & 2498 & 5170 & 7975 & 9334 & 8348 & 5686 & 2903 & 1075 & 272 & 42 & 3  \\
\end{tabular} 
\end{tiny}
\caption{The Betti table for $R/I_G$ with $G$ sampled from $G(n,n^{-0.27})$ with $n=20$.  All Betti numbers that are not indicated are zero.}\label{sampleBettiTable}
\end{table}

Although the Krull dimension of $R/I_G$ is fixed and finite in the regime $p=n^{-\alpha}$, it is not hard to see that the projective dimension grows linearly with $n$.  Furthermore, the results of Erman and Yang discussed above imply that if we take $G \sim G(n,p)$ with $p = n^{-\alpha}$ and $1/(d + 1) < \alpha < 1/d$ then almost every entry in the first $d + 1$ rows of the Betti table of $R/I_G$ is nonzero.   We will be interested in the entries below row $d+1$, as well as the shape of the Betti table of $R/I_G$ on the ``right side''.

We first turn to the vanishing of Betti numbers below row $d + 1$ and in particular the \emph{regularity of $R/I_G$}, by definition the largest row of the Betti table that has a nonzero entry. Results from \cite{ErmYan} imply that if $p = n^{-\alpha}$ for $1/(d + 1) < \alpha < 1/d$ then the regularity of $R/I_G$ satisfies
\[d+1 \leq \reg(R/I_G) \leq 2d.\]
Erman and Yang ask if this bound can be improved \cite[Question 5.3]{ErmYan}, and our next result gives a precise answer for any choice of coefficient field. We extend a result of Malen \cite{Malen} regarding collapsibility for random clique complexes to prove that with high probability all induced complexes of $\Delta(n,n^{-\alpha})$ have vanishing integral homology in dimension larger than $d$ (see Section \ref{sec:proj}). As a result we obtain the following.

\begin{theorem}\label{RegularityTheorem}
Let $d$ be any positive integer, fix $\alpha$ with $1/(d+1) < \alpha < 1/d$, and let  $G \sim G(n, n^{-\alpha})$. Then for any field ${\kk}$ of coefficients, with high probability we have  \[\reg(R/I_G) = d + 1.\] 
\end{theorem}


Next we turn to the right-hand side of the Betti table of a random coedge ideal.  Here we consider the \emph{projective dimension} of $R/I_G$, by definition the largest column of the Betti table that has a nonzero entry. Hilbert's Syzygy Theorem provides the upper bound $\pdim(R/I_G) \leq n$ and knowledge of the Krull dimension of $R/I_G$ provides a trivial lower bound of $n - (2d + 2)$ for any choice of coefficients.  As we discuss in Section \ref{sec:proj}, an easy application of Hochster's formula and known results about random clique complexes in fact implies that $\pdim(R/I_G) \geq n - (d+1)$. We are able to sharpen the upper bound for arbitrary coefficients and provide the precise value of $\pdim(R/I_G)$ for the case $\kk = {\mathbb Q}$.



\begin{theorem}\label{PdimTheorem}
Suppose $p = n^{-\alpha}$ with $1/(d+1) < \alpha < 1/d$ and let $G = G(n,p)$. Then the projective dimension of $R/I_G$ can be bounded as follows.

\begin{enumerate}
    \item For any coefficient field $\kk$ we have with high probability
\[n-(d+1) \leq \pdim(R/I_G) \leq n-(d/2).\]
\item
For the case $\kk = {\mathbb Q}$ we have with high probability
\[\pdim(R/I_G) = n-(d+1).\]
\end{enumerate}
\end{theorem}

We prove Theorem \ref{PdimTheorem} by relating $\pdim(R/I_G)$ to a family of graph parameters $\kappa_{\kk}^i$, first introduced by Babson and Welker in \cite{BabWel}.  These parameters involve the homology of subcomplexes of $\Delta(G)$ obtained by removing at most $\ell$ vertices, generalizing the notion of an $\ell$-vertex-connected graph. We discuss these ideas in Section \ref{sec:combinatorialdescriptions}.  The calculation of $\kappa_{\kk}^i$ for the case $\kk = {\mathbb Q}$ involves spectral graph theory and an application of Garland's method.  We review those notions in Section \ref{sec:spectral}.

By the Auslander-Buchsbaum formula, our bounds on projective dimension of $R/I_G$ in turn lead to bounds on the \emph{depth} of $R/I_G$.  In particular we see that $\depth(R/I_G) \leq d+1$ for any field of coefficients, with equality in the case that $\kk = {\mathbb Q}$. This is turn gives a good description of the (lack of) Cohen--Macaulay properties of $R/I_G$.  In the following we use $\dim(M)$ to denote the \emph{Krull dimension} of an $R$-module $M$.

\begin{corollary}\label{cor:depthCM}
For $d \geq 1$ suppose $p = n^{-\alpha}$ with $1/(d+1) < \alpha < 1/d$, and let $G = G(n,p)$. Then for any coefficient field $\kk$, with high probability we have
\[\depth(R/I_G) \leq \left \lceil \frac{\dim(R/I_G)}{2} \right \rceil.\] 

\noindent
In particular, with high probability $R/I_G$ is not Cohen--Macaulay over any field.
\end{corollary}

This complements results from \cite{ErmYan} where Cohen--Macaulay properties were considered for the regime $1/n^{2/3} < p < \large(\frac{\log n}{n}\large)^{2/k+3}$.

Our results also lead to an understanding of the extremal Betti numbers of random coedge ideals.  Recall that a nonzero graded Betti number $\beta_{i,i+k}(R/I)$ is said to be \emph{extremal} if $\beta_{i',i'+k'}(R/I) = 0$ for all pairs $(i',k') \neq (i,k)$ with $i' \geq i$ and $k' \geq k$. The extremal Betti numbers detect the ``corners'' in the Betti table of $R/I_G$ and can be seen as a generalization of regularity and projective dimension. In \cite[Theorem 1]{HKM} it is shown that for any integer $r$ and $b$ with $1 \leq b \leq r$ there exists a graph $G$ for which $\reg(R/I_G) = r$ and such that $R/I_G$ has $b$ extremal Betti numbers.  As a corollary of our results we see that for the case of random coedge ideals the extremal Betti numbers are very easy to describe.

\begin{corollary}\label{cor:extremal}
Suppose $p = n^{-\alpha}$ with $1/(d+1) < \alpha < 1/d$, and let $G = G(n,p)$. Then for $\kk = {\mathbb Q}$ with high probability the ring $R/I_G$ has only one extremal Betti number, given by $\beta_{n - (d + 1), n}(R/I_G)$.
\end{corollary}

\begin{figure}
\centering 
\includegraphics[width=6 in]{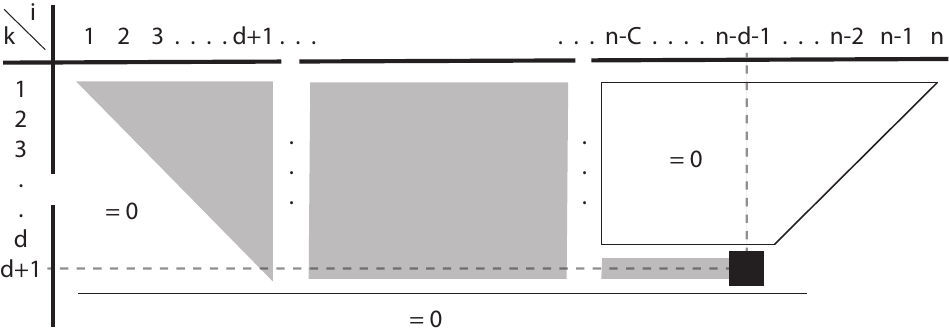}
\caption{The Betti table for $R/I_G$ with rational coefficients, where $G \sim G(n, n^{-\alpha})$ for $1/(d + 1) < \alpha < 1/d$. Theorem \ref{PdimTheorem} establishes the trapezoid of zero entries on the right, and Theorem \ref{RegularityTheorem} establishes zero entries below row $d + 1$. The black square indicates the single extremal Betti number of $R/I_G$, as described in Corollary. \ref{cor:extremal}}\label{ResultsSummary}
\end{figure}





For our last result we ``zoom out'' and consider the distribution of nonzero entries in the Betti table of $R/I_G$ for the case ${\mathbb K} = {\mathbb Q}$. Erman and Yang consider this question for the case of $p=c/n$, for some fixed constant $c$, and in \cite[Theorem 1.4]{ErmYan} they show that if $c \in (0, 1)$  then $\overline{\beta}_{i, i + 1}$ converges to $Kn$ for an explicit constant $K = K(c)$ for $i \approx n/2$.  Hence the distribution of entries in the first row of the Betti table of $R/I_G$ looks like those of a Koszul complex near $n/2$. Furthermore the authors conjecture (\cite[Conjecture 6.4]{ErmYan}) that a similar property should hold for other rows of the Betti table for larger values of $p$.  

Here we take a somewhat different perspective to describe the distribution of nonzero Betti numbers for a wider range of parameters $p$, and for all rows of the Betti table. In light of Hochster's formula (Theorem \ref{thm:Hochster}) it is natural to consider a normalized version of the Betti number defined by:
\[\overline \beta_{i,j}(R/I_G) = \frac{\beta_{i,j}(R/I_G)}{\binom{n}{j}}\]
We consider homology of subcomplexes in the $p = n^{-\alpha}$ regime and see that for rational coefficients we obtain a ``logarithmic diagonal'' distribution among the nonzero normalized Betti numbers. More specifically we have the following result.


\begin{theorem}\label{corstaircase}
Let $d \geq 0$ be any integer, fix $\alpha \in (1/(d + 1), 1/d)$, and let $G \sim G(n, n^{-\alpha})$. Then for $\kk = {\mathbb Q}$ we have for any $k \leq (d + 1)$ and $i = n^\gamma$
\[\overline \beta_{i, i + k}(R/I_G) \rightarrow \infty\]
in probability if $\gamma \in (0, 1) \cap ((k - 1) \alpha, k \alpha)$,
while 
\[\overline \beta_{i, i + k}(R/I_G) \rightarrow 0\]
in probability if $\gamma \in (0, 1) \setminus [(k - 1) \alpha, k \alpha]$.
\end{theorem}


Table \ref{normalizedSample} depicts the normalized Betti table for the example from Table \ref{sampleBettiTable}, and we refer to Figure \ref{summary} for an illustration of the asymptotic results from Theorem \ref{corstaircase}. 


\begin{table}[h]
\begin{tiny}
\begin{tabular}{c|cccccccccccccccccccc}
- & 0 & 1 & 2 & 3 & 4 & 5 & 6 & 7 & 8 & 9 & 10 & 11 & 12 & 13 & 14 & 15 & 16 & 17 & 18 & 19  \\ \hline
0 & 1 & - & - & -& - & - & -& - & - & -& - & - & -& - & - & - & - & - & - & - \\
1 & - & 0.56 & 0.76 & 0.72 & 0.57 &  0.40 & 0.26 & 0.17 & 0.11 & 0.07 & 0.04 & 0.02 & 0.01 & 0.01 & 0.00 & 0.00 & 0 & 0 & 0 & 0 \\
2 & -  & - & 0.04 & 0.16 & 0.40 & 0.75 & 1.17 & 1.61 & 2.01 & 2.32 & 2.53 & 2.61 & 2.56 & 2.36 & 2.04 & 1.61 & 1.11 & 0.55 & 0 & - \\
3 & -  & - & - & 0.00 & 0.00 & 0.00 &  0.01 & 0.01 & 0.03 & 0.06 & 0.12 & 0.22 & 0.37 & 0.60 & 0.94 & 1.43 & 2.1 & 3 & - & - \\
\end{tabular}
\end{tiny}
\caption{The normalized Betti table for $R/I_G$ from Example \ref{ex:running}, obtained by dividing the Betti number $\beta_{i,i+k}$ by $\binom{20}{i+k}$.}\label{normalizedSample}
\end{table}





Finally we mention some previous work that is related to ours.  In \cite{DPSSW} De Loera, Petrovı\'c, Silverstein, Stasi, and Wilburne study a notion of random monomial ideals of arbitrary uniform degree, and establish several results and conjectures regarding their algebraic invariants.  Their construction is similar to that of Erd\H{o}s--R\'enyi random graphs, and under certain parameters can recover our model.  In a recent paper of Banerjee and Yogeshwaran \cite{BanYog} it is shown that the threshold for a coedge ideal to have a linear resolution coincides with the threshold for the ideal to have a linear presentation.  They also establish precise threshold results regarding regularity for the case that $\reg(I_G) = 2$.  In \cite{DHKS} De Loera, Ho\c{s}ten, Krone, and Silverstein study resolutions of random equigenerated monomial ideals under a model where the number of variables $n$ is fixed and the degree $d$ of the randomly chosen generators goes to infinity.  In this model they prove that the projective dimension of such ideals is as large as possible (equal to $n$). In \cite{ConJuhWel} Conca, Juhnke-Kubitzke, and Welker study a model for asymptotic Betti numbers via the Stanley--Reisner ring of successive barycentric subdivisions of a simplicial complex. More recently Booms, Erman, and Yang \cite{BooErmYan} study how the underlying field $\kk$ plays a role in the Betti tables of random ideals.  They prove that with high probability the Betti numbers of $R/I_G$ depend on the characteristic of $\kk$.  Finally, in a very recent paper \cite{EngOrl} Engstr\"om and Orlich apply the notion of `critical graphs' from random graph theory to study `parabolic' Betti numbers of random edge ideals.

The rest of the paper is as organized follows.  In Section \ref{sec:background} we review some background material from combinatorial commutative algebra, random graphs, and random cliques complexes.  Here we also provide combinatorial descriptions of the relevant algebraic invariants that will be useful for our later study, including the notion of cohomological vertex connectivity $\kappa_\kk^i$ discussed above. In Section \ref{sec:regularity} we focus on the regularity of $R/I_G$ for random coedge ideals over any field $\kk$, and provide the proof of Theorem \ref{RegularityTheorem}.  Here we also establish Lemma \ref{lemma2ndmomentbelow}, which provides an upper bound on the expectation of higher Betti numbers of relevant clique complexes. In Section \ref{sec:proj} we consider the projective dimension of $R/I_G$. We first consider the case of $\kk = {\mathbb Q}$ where we use methods from spectral graph theory to provide a proof of Theorem \ref{PdimTheorem} part (1).  We next turn to the proof of part (2), where different arguments are needed.  We end Section \ref{sec:proj} with an application to the depth and extremal Betti numbers of $R/I_G$ and prove Corollaries \ref{cor:depthCM} and \ref{cor:extremal}.  In Section \ref{sec:log} we consider the distribution of nonzero entries in each row of the Betti table of $R/I_G$ and prove Theorem \ref{corstaircase}.  Finally in Section \ref{sec:further} we discuss some open questions and ideas for future research.

{\bf Acknowledgments.} We thank the anonymous referees for helpful comments and corrections.

\section{Definitions and preliminaries} \label{sec:background}

In this section we recall some necessary definitions and set some notation that will used throughout the paper.

\subsection{Commutative algebra and Hochster's formula}
Let $G = (V,E)$ be a simple undirected graph on vertex set $V = [n] = \{1,\dots,n\}$ and edge set $E = E(G)$, and as above let $R = \kk[x_1, \dots, x_n]$ be the polynomial ring on $n$ variables over some fixed field $\kk$. The \defi{coedge ideal} $I_G$ of $G$ is defined to be the ideal generated by all monomials corresponding to non-edges of $G$:
\[ I_G = \langle x_ix_j : ij \notin E(G) \rangle. \]

We note that the collection of coedge ideals correspond to the set of squarefree quadratic monomial ideals in $R$, but it will be useful to have the graph theoretic perspective.  Furthermore, a coedge ideal is a special case of the more general \emph{Stanley--Reisner ideal} of a simplicial complex (see \cite{Sta}, where much of the following material is also discussed).  In particular $I_G$ is the Stanley--Reisner ideal of $\Delta(G)$, the \defi{clique complex} of the graph $G$.  Recall that $\Delta(G)$ is the simplicial complex on vertex set $V$ with simplices given by all complete subgraphs of $G$.  The Stanley--Reisner perspective will be useful when we talk about homological properties of $R/I_G$.  

A \defi{free resolution} of $R/I_G$ is a long exact sequence of $R$-modules \[  0 \ \leftarrow \ R/I_G \ \leftarrow \ R \ \xleftarrow{\phi_1} \ F_1 \ \xleftarrow{\phi_2} 
    \ \cdots \ \xleftarrow{\phi_r} \ F_r \ \leftarrow \ 0,
\]
where each $F_i$ is a graded free $R$-module
\[
    F_i \ \cong \ \bigoplus_{\sigma \in \N^n} R(-\sigma)^{\beta_{i,\sigma}}
\]
and where each $\phi_i$ is a graded homomorphism of $R$-modules.

The resolution is called \defi{minimal} if each of the $\beta_{i,\sigma}$ is
minimum among all graded free resolutions of $R/I_G$. In this case we have
$\beta_{i,\sigma} = \beta_{i, \sigma}(R/I_G) = \Tor_R(R/I_G, \kk)_\sigma$, and these integers are called the \defi{finely} or
\defi{$\N^n$-graded Betti numbers} of the $R$-module $R/I_G$. 

The main objects we study here are the \defi{$\Z$-graded Betti numbers} of $R/I_G$, given by
\[
    \beta_{i,j}(R/I_G) \ = \ \sum_{|\sigma| = j} \beta_{i,\sigma}(R/I_G), \]
\noindent
where $|\sigma| = \sigma_1 + \sigma_2 + \cdots + \sigma_n$.

Note that $I_G$ is minimally generated by $\binom{n}{2} - |E(G)|$ monomials, and hence
\[\beta_{1,2}(R/I_G) = \binom{n}{2} - |E(G)|. \]

It is often useful to collect the Betti numbers of $R/I_G$ in its \defi{Betti table}, where the entry in row $k$ and column $i$ has the value $\beta_{i,i+k}(R/I_G)$.  Note that by construction $\beta_{i,j}(R/I_G) = 0$ if $j<i$. We will be interested in the ``shape'' of these Betti tables for $R/I_G$ where $G$ is an Erd\H{o}s--R\'{e}nyi random graph. The following invariants will be useful in our study.

The \defi{projective dimension} of $R/I_G$ is given by
\[\pdim(R/I_G) = \max\{i: \beta_{i,i+k}(R/I_G) \neq 0 \text{ for some $k$}\},\]
whereas the \defi{(Castelnuovo-Mumford) regularity} of $R/I_G$ is given by
\[\reg(R/I_G) = \max\{k: \beta_{i,i+k}(R/I_G) \neq 0 \text{ for some $i$}\}.\]


 Following \cite{BCP} a nonzero graded Betti number $\beta_{i,i+k}(R/I_G)$ is said to be \defi{extremal} if $\beta_{i',i'+k'}(R/I_G) = 0$ for all pairs $(i',k') \neq (i,k)$ with $i' \geq i$ and $k' \geq k$. Hence extremal Betti numbers detect the ``corners'' in the Betti table of $R/I_G$ and can be seen as a refinement of the notion of regularity  and projective dimension.


Our main tool for studying Betti numbers will be a well known formula due to Hochster (see \cite[Chapter II, Theorem 4.1]{Sta}, where an equivalent formulation in terms of \emph{links} is given).  In what follows recall that $\Delta(G)$ denotes the clique complex of $G$.

\begin{theorem}[Hochster's formula]\label{thm:Hochster}
For a graph $G$, the Betti numbers of $R/I_G$ are given by
\[\beta_{i,j}(R/I_G) = \sum_{S \in {[n] \choose j}} \dim_{\kk} \big(\tilde H_{j-i-1} (\Delta(S); \kk ) \big), \]
where $\Delta(S)$ denotes the clique complex on the graph induced by the subset $S$.
\end{theorem}

Hochster's formula provides a justification for the standard way of writing the Betti table, where the entry in row $k$ and column $i$ is given by $\beta_{i, i + k}(R/I_G)$.  In particular the entry in row $k$ and column $i$ gives the sum of the $(k - 1)$st topological Betti numbers over all induced subcomplexes on $i + k$ vertices. 
The first row of the Betti table then summarizes the zeroth homology group of subcomplexes, the second row summarizes the first homology group of subcomplexes, and so on. In Table \ref{sampleBettiTable}, we give an example of such a Betti table for $R/I_G$ and $G$ an instance of $G(n, p)$.

\subsection{Combinatorial descriptions of the invariants}\label{sec:combinatorialdescriptions}

In the context of coedge ideals $I_G$, many algebraic invariants of $R/I_G$ can be described in terms of combinatorial properties of the underlying graph $G$.   For instance recall that the \defi{Krull dimension} of a commutative ring is the supremum length of chains of prime ideals. For the case of coedge ideals $I_G$, it is known (see for instance \cite[Chapter II, Theorem 1.3]{Sta}) that the Krull dimension $\dim(R/I_G)$ is given by the clique number of $G$, i.e. one more than the dimension of $\Delta(G)$. In a similar spirit, Hochster's formula provides a way to compute the Betti numbers of $R/I_G$ via the topology of the clique complex of the underlying graph $G$.  To establish our results it will be useful to translate relevant algebraic properties of $R/I_G$ in these terms. 

 In particular we can define the regularity of $R/I_G$ in terms of homology of induced subcomplexes of $\Delta(G)$ as follows. By Hochster's formula we have
\[\reg(R/I_G) = \max \{k : H_{k - 1}(\Delta(S); \kk) \neq 0 \text{ for some $S \subseteq [n]$}\}.\]
In the example in Table \ref{sampleBettiTable} we see that $\beta_{17,20}(R/I_G) \neq 0$, and since $n=20$ this implies that the entire complex $\Delta(G)$ has nonzero homology in degree 2.  From the table we also see no nonzero entries in rows larger than 3, which implies that $\reg(R/I_G) = 3$ and in particular no induced subcomplex of $\Delta(G)$ has homology in degree larger than 2. We show in Theorem \ref{RegularityTheorem} that for random graphs $G \sim G(n, p)$ with $p$ in a certain regime this is the typical behavior; with high probability the largest nonvanishing homology group of the entire clique complex $\Delta(G)$ gives the regularity of $R/I_G$. 
This is of course not always the case even for coedge ideals.  For example if we let $W_n$ denote the \emph{wheel graph} (the cone over a $n$-cycle $C_n$) then for $n \geq 4$ we have that $\Delta(W_n)$ is contractible, whereas $C_n$ is an induced subcomplex which has homology in degree 1.  Hence the regularity of $R/I_{W_n}$ over any field $\kk$ is 2.

The projective dimension of $R/I_G$ can likewise be interpreted from the topology of subcomplexes of $\Delta(G)$. For this we first recall a standard notion from graph theory. A graph $G$ on more than $\ell$ vertices is said to be \defi{$\ell$-connected} if for every set $S \subset G$ of size less than $\ell$, the graph $G \setminus S$ is connected. The \defi{vertex connectivity} of $G$ is the maximum $\ell$ such that $G$ is $\ell$-connected.

From the definition of the projective dimension and Hochster's formula we have that if the vertex connectivity of $G$ is $\ell$ then the projective dimension of $R/I_G$ over any field $\kk$ is at least $n - \ell - 1$. Indeed in this case $G$ has a set of $\ell$ vertices whose deletion disconnects the graph, and hence we have obtain a set $S \subseteq G$ with $|S| = n - \ell$ where $\tilde H_0(\Delta(S);{\mathbb Z}) \neq 0$ (i.e. $\Delta(S)$ is disconnected). By Hochster's formula we get $\beta_{n - \ell - 1, n - \ell}(R/I_G) \geq 1$, so for $k = 1$ and $i = n - \ell - 1$, we have $\beta_{i, i + k}(R/I_G) \neq 0$, providing the desired lower bound on projective dimension.

Recall that the projective dimension is defined as the index of the last nonzero column of the Betti table. Motivated by the notion of vertex connectivity of a graph we can define projective dimension in terms of a topological and combinatorial invariant of $\Delta(G)$. This approach is taken in work of Babson and Welker \cite{BabWel}, and borrowing their notation we let $\kappa^i_{\kk}(G)$ denote the \defi{$i$-cohomological vertex connectivity} of $G$ over the field $\kk$, where 
\[\kappa^i_{\kk}(G) := \min \{|T| : T \subseteq V(G), \tilde{H}^{i}(\Delta(V(G) \setminus T); \kk)) \neq 0 \}. \]

We set $\kappa^i_{\kk}(G) = \infty$ if no such $T$ exists. Note that for any field $\kk$ we have that $\kappa^0_{\kk}(G)$ recovers the usual vertex connectivity of $G$. For a graph $G$ on $n$ vertices the projective dimension of $R/I_G$ over $\kk$ can then be expressed in terms of these connectivity parameters as follows.
\begin{lemma}\label{combinatorialpdim}
For any graph $G$ and any field $\kk$, we have 
\[\pdim(R/I_G) = \max_{i} \{n - i - \kappa_{\kk}^{i - 1}(G)\}.\]
\end{lemma}
\begin{proof} From the definition of projective dimension, we have that $\pdim(R/I_G)$ is the maximum index $i$ so that there exists $k$ with $\beta_{i, i+k}(R/I_G) \neq 0$. Therefore, if $\pdim(R/I_G) = m$ then there exists a maximal set of vertices $S$ of size $s$ for some $s > 0$ so that $\tilde{H}_{s - m - 1}(\Delta(S); \kk) \neq 0$. Thus we have found a set $T = [n] \backslash S$ of size $n-s$ where $\tilde H_{s-m-1}(\Delta(V(T) \setminus T); \kk) \neq 0$ so that $\kappa_{\kk}^{s - m - 1} \leq n - s$. Hence for $i = s - m$ we have $n - i - \kappa_{\kk}^{i - 1} \geq n - (s - m) - (n - s) = m = \pdim(R/I_G)$.

Conversely, if $\max_i \{n - i - \kappa_{\kk}^{i - 1}(G)\} = m$, then there exists $s$ so that $\kappa^{s - 1}_{\kk}(G) = n - s - m$.  Therefore $G$ admits a collection $T$ of $n - s - m$ vertices where $\tilde{H}_{s - 1}(\Delta(V(G) \setminus T); \kk) \neq 0$, and we conclude that $\beta_{m, m+s}(R/I_G) \neq 0$.  By definition the implies that $\pdim(R/I_G) \geq m = \max_i \{n -  i - \kappa_{\kk}^{i - 1}(G)\}$.
\end{proof}

In the example in Table \ref{sampleBettiTable} we have $\kappa_{\Q}^0(G) = 4$, $\kappa_{\Q}^{1}(G) = 1$, and $\kappa_{\Q}^{2}(G) = 0$. From this it follows that $\pdim(R/I_G) = 17$.


\subsection{Random graphs and random clique complexes}\label{sec:randomgraphs}
 The study of Erd\H{o}s--R\'enyi random graphs is a well-established research area with many classical results. Here we collect some results on random graphs that will be useful for our study, and refer to \cite{FriKar} for an extensive overview. In light of Hochster's formula we are especially interested in thresholds for notions of connectivity and induced subgraphs.

As we describe asymptotic behavior of our objects we often make use of Bachmann--Landau notation. Specifically given two nonnegative functions $f(n)$ and $g(n)$ we use the following notation:
\begin{itemize}
    \item $f(n) = O(g(n))$ provided there exists $C$ constant so that $f(n) \leq C g(n)$ when $n$ is sufficiently large.
    \item $f(n) = \Omega(g(n))$ provided that there exists $c$ constant so that $f(n) \geq cg(n)$ when $n$ is sufficiently large.
    \item $f(n) = \Theta(g(n))$ provided that $f(n) = O(g(n))$ and $f(n) = \Omega(g(n))$ both hold.
    \item $f(n) = o(g(n))$ provided that $\lim_{n \rightarrow \infty} f(n)/g(n) = 0$.
    \item $f(n) = \omega(g(n))$ provided that $\lim_{n \rightarrow \infty} f(n)/g(n) = \infty$.
\end{itemize}

A seminal result from \cite{ErdRen} addresses the connectivity threshold of the random graph $G(n,p)$. For $\epsilon > 0$, if $p(n) < \frac{(1-\epsilon) \log n}{n}$ then a graph $G \sim G(n,p)$ will almost surely contain an isolated vertex (and hence will be disconnected), whereas if $p(n) > \frac{(1+\epsilon) \log n}{n}$ then a graph $G \sim G(n,p)$ will almost surely be connected. 

In a similar vein, vertex $\ell$-connectivity of a random Erd\H{o}s--R\'enyi graph is also well understood. As we saw in the previous section this notion is relevant for our study of projective dimension of random coedge ideals.  In \cite{ErdRen64} it is shown that the threshold for $\ell$-connectivity of $G(n,p)$ is given by 
\[p = \frac{\log n + (\ell-1) \log \log n + c_n}{n}.\]
\noindent
It is well-known that this coincides with the threshold for $G(n,p)$ to have minimum degree $\ell$.  It is clear that if $G$ has minimum degree $\ell$ then one can remove $\ell$ vertices to disconnect the graph; the surprising result is the other direction. 

As we mentioned above, Babson and Welker have studied a higher-dimensional notion of vertex connectivity for clique complexes.   In \cite{BabWel} they mostly study the case that $\kk = {\mathbb F}_2$ and they establish sufficient conditions on $p$ for a random clique complex to have  $i$-cohomological vertex connectivity $\kappa_{\mathbb{F}_2}^i$ equal to the minimum degree of an $i$-dimensional face.  We say more about this in Section \ref{sec:proj}.

For our study it will also be useful to understand thresholds for the appearance of certain induced subgraphs of $G \sim G(n, p)$.  For example it is well known that the \emph{clique number} $\omega(G)$ (the size of the largest complete subgraph) of $G \sim G(n, n^{-\alpha})$ is asymptotically almost surely a constant depending on $\alpha$.  This follows from a more general result of Bollob\'as \cite{Bol} that we will use again later. For a graph $H$ the \defi{essential density} of $H$ is defined to be 
\[\nu(H)  = \max \{e(H')/v(H') \mid H' \subseteq H \}\]
Bollob\'as' result, in part, is the following. The statement as it appears here is essentially \cite[Theorem 5.3]{FriKar}.
\begin{proposition}[\cite{Bol}]\label{prop:density}
Suppose $H$ is a fixed graph.  If $\alpha < 1/\nu(H)$ then with high probability $G \sim G(n, n^{-\alpha})$ contains $H$ as a subgraph, whereas if  $\alpha > 1/\nu(H)$ then with high probability $G \sim G(n, n^{-\alpha})$ does not contain $H$ as a subgraph.
\end{proposition}

It is easy to see that the essential density of the complete graph $K_{m + 1}$ is $\frac{m}{2}$, from which the following result follows. 

\begin{corollary}\label{cor:cliques}
Suppose $G \sim G(n,p)$ for $p = n^{-\alpha}$.  Then with  high probability we have that $\omega(G) =2d + 1$ if $2/(2d + 1) < \alpha < 1/d$ and $\omega(G) = 2d + 2$ if $1/(d + 1) < \alpha < 2/(2d + 1)$.
\end{corollary}

\begin{remark}
If $p = cn^{-(2/(d + 1))}$ for some constant $c$, then the number of $2d + 2$ cliques in $G(n, p)$ can be shown to be Poisson distributed, see for example \cite[Theorem 5.4]{FriKar}. 
\end{remark}

In recent years the techniques of random graphs have been generalized to higher dimensional settings. The \defi{random clique complex} is a model for random simplicial complexes based on the Erd\H{o}s--R\'{e}nyi random graph, first introduced by Kahle \cite{Kah}. One samples a simplicial complex $\Delta$ from this model $\Delta(n, p)$ by taking the clique complex $\Delta(G)$ of an Erd\H{o}s--R\'{e}nyi graph $G \sim G(n, p)$. 

Here we will be considering reduced $i$-dimensional homology groups of a simplicial complex $\Delta$ over various coefficient rings $A$, which we denote by $\tilde H_i(\Delta, A)$. When $A = {\mathbb Z}$ we use $\tilde H(\Delta)$ to denote the \emph{(reduced) integral homology} of $\Delta$.  When $A = {\mathbb Q}$, the \emph{rational homology} of $\Delta$ is a vector space over $\mathbb Q$, and we let $\beta_i(\Delta) = \beta_i(\Delta;{\mathbb Q})$ denote the (topological) \defi{Betti numbers} of $\Delta$, given by the dimension of $\tilde H(\Delta;{\mathbb Q})$.  We emphasize that $\beta_{i,j}(R/I_G)$ and $\beta_i(\Delta)$ mean different things, but the difference will be clear in context.

In this context thresholds for vanishing of homology with specified coefficients provide the natural generalization of connectedness for random graphs.  In particular one of the fundamental results from the study of random clique complexes, which we will need in our study, is the the following result due to Kahle.

\begin{theorem}\cite[Corollary 1.3]{Kah2014}\label{Kah2014}
Suppose $d \in \N$ is fixed and $\alpha \in {\mathbb R}$ satisfies $1/(d + 1) < \alpha < 1/d$. Then with high probability the random clique complex $\Delta \sim \Delta(n, n^{-\alpha})$ satisfies the following:
\begin{itemize}
    \item $\tilde{H}_d(\Delta; \Q) \neq 0$, and 
    \item For $i \neq d$, $\tilde{H}_{i}(\Delta; \Q) = 0$.
\end{itemize}
\end{theorem}
The proof of Kahle's result above depends heavily on the fact that homology is computed with rational coefficients. In particular Kahle's proof of the second part of Theorem \ref{Kah2014} uses \emph{Garland's method}, which provides a way to study homology with coefficients in a field of characteristic zero. An important conjecture in stochastic topology is Kahle's bouquet of spheres conjecture:
\begin{conjecture}\cite[Conjecture 5.2]{Kah2014} \label{conj:spheres}
For $d \geq 3$ and $1/(d + 1) < \alpha < 1/d$ with high probability $\Delta \sim \Delta(n, n^{-\alpha})$ is homotopy equivalent to a bouquet of $d$-dimensional spheres.
\end{conjecture}
We note that $d \geq 3$ is necessary since results of Babson \cite{Bab} and independently Costa, Farber, and Horak \cite{CosFarHor} show that $\Delta \sim \Delta(n, n^{-\alpha})$ is simply connected for $\alpha < 1/3$ and not simply connected for $\alpha > 1/3$. In addition, with this knowledge of the fundamental group the Hurewicz theorem implies that Conjecture \ref{conj:spheres} will be proved if one can strengthen Theorem \ref{Kah2014} to work when $\Q$ is replaced by $\Z$. However at this point we only have a limited understanding about the integer homology of random clique complexes.

More recently Kahle's results have been sharpened by employing tools from simple homotopy theory. We will need these methods for our study and we review the basics here. We say that a face $\sigma$ of a simplicial complex is \defi{free} provided it is contained in exactly one proper coface i.e. a distinct face that properly contains $\sigma$. An \defi{elementary collapse} on a simplicial complex is the removal of a free face along with its unique coface. Such a collapse is a homotopy equivalence.  A simplicial complex is said to be \defi{collapsible} provided there is a sequence of elementary collapses that reduce it to a single vertex. For $i \geq 1$, a complex is said to be \defi{$i$-collapsible} provided there is a sequence of elementary collapses so that the resulting complex is at most $(i - 1)$-dimensional.  In particular an $i$-collapsible complex has no homology in degrees at least $i$, but the converse need not hold if $i \neq 1$. For $i \geq 2$ there are many $i$-acyclic complexes which fail to be $i$-collapsible. 

In \cite{Malen}, Malen proves the following theorem regarding $i$-collapsibility of random clique complexes.

\begin{theorem}\cite[Theorem 1.1]{Malen}\label{thm:collapsible}
Fix an integer $d \geq 0$.  If $\alpha > 1/(d + 1)$ then with high probability the complex $\Delta(n, n^{-\alpha})$ is $(d + 1)$-collapsible.
\end{theorem}

In the results described above we typically consider the random clique complex $\Delta(n, n^{-\alpha})$ for $\alpha$ fixed. From now on unless otherwise specified we will fix $\alpha$ to be between $1/(d + 1)$ and $1/d$, for $d$ a positive integer.  With this choice of $p$ we have from Corollary \ref{cor:cliques} that the clique number $\omega(G)$ takes on one of two values with high probability.  Since the Krull dimension of $R/I_G$ is given by $\omega(G)$ we have the following observation.

\begin{corollary}\label{cor:Krulldim}
Suppose $p = n^{-\alpha}$ with $1/(d+1) < \alpha < 1/d$ for some $d \geq 1$, and let $G = G(n,p)$. Then with high probability the Krull dimension of $R/I_G$ is either $2d + 1$ or $2d + 2$. More specifically the dimension is $2d + 1$ if $2/(2d + 1) < \alpha < 1/d$ and it is $2d + 2$ if $1/(d + 1) < \alpha < 2/(2d + 1)$.
\end{corollary}


\section{Regularity}\label{sec:regularity}

\noindent
In this section we prove Theorem \ref{RegularityTheorem}, which says that if $G \sim G(n,n^{-\alpha})$ for $1/(d + 1) < \alpha < 1/d$ then for any choice of coefficient field $\kk$ we have 
\[\reg(R/I_G) = d+1.\]

To place this result in context it should first be observed that bounding the regularity of arbitrary ideals in terms of $n$, the number of generators, cannot be too fruitful.  Examples dues to Mayr and Meyer \cite{MayMey} show that even for ideals generated by quadratic binomials in $n$ variables, one can achieve regularity on the order of $2^{2^n}$.  On the other hand, for coedge ideals $I_G$ (or more generally monomial ideals generated in degree 2) it can be shown that $\reg(R/I_G) \leq n$.


For the case of random coedge ideals, as usual we suppose $p=n^{-\alpha}$ with $1/(d + 1) < \alpha < 1/d$.  As a warmup to our main result we first observe that it is not hard to establish that the bounds
\[ d+1 \leq \reg(R/I_G) \leq 2d+2\]
hold with high probability for any $G \sim G(n,p)$. 
Indeed, for this choice of $p$ we have from Corollary \ref{cor:cliques} that the clique number of $G$ (and hence of any induced subgraph) is at most $2d+2$ with high probability. Hence no subcomplex of $\Delta(n,p)$ can have homology in degree larger than $2d+1$, and by Hochster's formula this provides the desired upper bound on the $\reg(R/I_G)$. On the other hand, we have from Theorem \ref{Kah2014} that $H_d(\Delta(G); \Q) \neq 0$ with high probability.  Hence another application of Hochster's formula implies that $\beta_{n-(d+1),n} \neq 0$ with high probability. By the universal coefficient theorem, nonvanishing homology over $\Q$ implies nonvanishing homology over any field. Thus for any choice of the coefficient field $\kk$ we have that $\reg(R/I_G) \geq d + 1$ with high probability.

Erman and Yang also make this observation in their work, and pose the question \cite[Question 5.3]{ErmYan} of whether or not the regularity of $R/I_G$ for $1/(d + 1) < \alpha < 1/d$ is exactly $d+1$ with high probability. One might suspect that this would be the case based on Kahle's result (Theorem \ref{Kah2014}) which says that for this choice of $p$ the random clique complex has nonzero rational homology only in degree $d$.  We remark that Erman and Yang do establish the vanishing of any $\beta_{i,i+k}$ for any \emph{fixed} $i$ and $k$ satisfying $k > d+1$ for this regime of $p$.  However this does not establish the desired bound on regularity as we need to consider all entries in these rows.

To prove that $R/I_G$ satisfies $\reg(R/I_G) \leq d + 1$, by Hochster's formula it suffices to prove that with high probability $\Delta \sim \Delta(n, n^{-\alpha})$ has no induced subcomplex with homology in degree larger than $d$. A natural way to prove this would be to bound the probability that a fixed set of vertices induces a subcomplex of $\Delta \sim \Delta(n, n^{-\alpha})$ with homology in degree larger than $d$, and then to sum over all subsets of $[n]$. This is not the approach we take, and instead we consider partial collapsiblity of random clique complexes.

 As mentioned above Malen has studied collapsibility properties of the random clique complex $\Delta \sim \Delta(n,n^{-\alpha})$ (see Theorem \ref{thm:collapsible} from Section \ref{sec:background}).  In order to prove his main result he establishes the following \emph{global} condition on a clique complex $\Delta$ which implies $(d + 1)$-collapbsiblity, and which also rules out the possibility of induced subcomplexes having homology in degree larger than $d$. In what follows, a pure $(d+1)$-dimensional simplicial complex is said to be \defi{strongly connected} if the ridge-facet graph is connected.



\begin{theorem}\cite[Theorem 3.1]{Malen}\label{DeterministicCollapsibility}
Fix $d \geq 0$. Let $\Delta$ be a finite clique complex such that every strongly connected, pure $(d + 1)$-dimensional subcomplex $S \subseteq \Delta$ contains at least one vertex $v$ with $\deg_S(v) \leq 2d + 1$. Then $\Delta$ is $(d + 1)$-collapsible. 
\end{theorem}

 Malen shows that if $\alpha > 1/(d+1)$ then with high probability the complex $\Delta \sim \Delta(n, n^{-\alpha})$ satisfies the hypothesis of Theorem \ref{DeterministicCollapsibility} (see Lemma 3.6 of \cite{Malen}), which leads to a proof of Theorem \ref{thm:collapsible}. The key insight here is that the hypothesis of Theorem \ref{DeterministicCollapsibility} is a condition on $\Delta$ which is closed under induced subcomplexes. This leads us to a short proof of Theorem \ref{RegularityTheorem}.

\begin{proof}[Proof of Theorem \ref{RegularityTheorem}]
We fix $\alpha$ with $1/(d+1) < \alpha$. We claim that if $\Delta \sim \Delta(n, n^{-\alpha})$ then with high probability every induced subcomplex of $\Delta$ satisfies the assumptions of Theorem \ref{DeterministicCollapsibility}. By Lemma 3.6 of \cite{Malen}, with high probability $\Delta$ itself satisfies the assumptions of Theorem \ref{DeterministicCollapsibility}, and so it suffices to verify that these assumptions are preserved by taking induced subcomplexes. 

For this suppose $\Delta$ is a flag complex that satisfies the assumptions of Theorem \ref{DeterministicCollapsibility}, and let $S$ be a subset of vertices of $\Delta$. We use $\Delta(S)$ to denote the flag complex induced on $S$.  Now suppose $Y$ is a strongly connected, pure $(d + 1)$-dimensional subcomplex of $\Delta(S)$.  We first observe that $Y$ is also a strongly connected, pure $(d + 1)$-dimensional subcomplex of $\Delta$. Therefore $Y$ contains a vertex $v$ with $\deg_Y(v) \leq 2d + 1$, and we have that $\Delta(S)$ satisfies the assumptions of Theorem \ref{DeterministicCollapsibility}.

Now $\Delta \sim \Delta(n, n^{-\alpha})$ satisfies the assumptions of Theorem \ref{DeterministicCollapsibility} with high probability, so with high probability every induced subcomplex of $\Delta$ also satisfies those assumption. This implies that every induced subcomplex of $\Delta$ is $(d + 1)$-collapsible, and in particular has no homology above degree $d$.  Thus $\reg(R/I_G) \leq d+ 1$ for $G \sim G(n, n^{-\alpha})$.

On the other hand if $\alpha$ is also bounded above by $1/d$ then by Theorem \ref{Kah2014}, $\Delta \sim \Delta(n, n^{-\alpha})$ has homology in degree $d$ so $\reg(R/I_G) \geq d + 1$ as well, and the result follows.
\end{proof}

Theorem \ref{RegularityTheorem} shows that if  $1/(d+1) < \alpha$ then with high probability no subcomplex of $\Delta \sim \Delta(n,n^{-\alpha})$ has homology in degree larger than $d$.  We end this section with a result that establishes the vanishing of the \emph{expectation} of higher Betti numbers in this regime. This fact will be used in the proof of Theorem \ref{corstaircase} given in Section \ref{sec:log}, but we include it here since it utilizes the same collapsibility arguments as in the proof of Theorem \ref{RegularityTheorem}. Note that this statement is stronger than simply the result that $\beta_{d + 1}(\Delta; {\mathbb Q}) = 0$ with high probability.

\begin{lemma}\label{lemma2ndmomentbelow}
Fix an integer $d \geq 0$ and let $1/(d + 1) < \alpha$. Then for $\Delta \sim \Delta(n, n^{-\alpha})$ we have
\[\E(\beta_{d + 1}(\Delta; \Q)) = o(1).\]
\end{lemma}

\begin{proof}
 We first note that for any complex $\Delta$, if $K$ is the support of a minimal $(d + 1)$-cycle in $H_{d + 1}(\Delta; \Q)$ then $K$ must be a pure, strongly connected $(d + 1)$-complex and have all vertices having degree at least $2d + 2$. Indeed if $\Delta$ contained a vertex of degree at most $2d + 1$ then by Theorem \ref{DeterministicCollapsibility}, $K$ would be $(d + 1)$-collapsible and in particular would  admit a free $d$-dimensional face, contradicting its minimality.  Hence we have that $\beta_{d + 1}(\Delta; \Q)$ is at most the number of pure, strongly connected $(d + 1)$-complexes with minimum vertex degree at least $2d + 2$ contained in $\Delta$.

Now let $\Delta \sim \Delta(n, n^{-\alpha})$ as above.  Our argument to establish the vanishing of $\E(\beta_{d + 1}(\Delta; \Q))$ splits into two parts. We first show that for any $M > 0$ with probability at least $1 - o(n^{-M})$ there are no large strongly connected $(d + 1)$-dimensional subcomplexes in $\Delta$. We then bound $\E(\beta_{d + 1}(\Delta; \Q))$ by enumerating the expected number of small strongly connected, pure $(d + 1)$-dimensional complexes with minimum vertex degree $2d + 2$.

The argument for ruling out large strongly connected subcomplexes first appeared in Lemma 5.1 of \cite{Kah}. Observe that the $(d + 1)$-dimensional faces of a pure, strongly connected $(d + 1)$-complex $K$ can be ordered $(\sigma_1, \sigma_2, ..., \sigma_m)$ so that each $\sigma_i$ intersects some previous $\sigma_j$ exactly in a $d$-dimensional face. This ordering then induces an ordering on the vertices $(v_1, .., v_{d + 2}, v_{d + 3}, ..., v_k)$ as follows.  We take $v_1, ..., v_{d + 2}$ to be the vertices of $\sigma_1$, and then add each further $v_i$ as it appears in $\sigma_1 \subseteq \sigma_1 \cup \sigma_2 \subseteq \cdots \subseteq \sigma_1 \cup \cdots \cup \sigma_m$. From the ordering on the simplices, each time a vertex is added (other than the first $d + 2$) exactly $d + 1$ edges are added. So if $K$ is to have exactly $d + 2 + L$ vertices then it has at least $\binom{d + 2}{2} + L(d + 1)$ edges. So for $K$ any strongly connected, pure $(d + 1)$-complex on $d + 2 + L$ vertices, the probability that $K$ is contained in $\Delta$ is at most
\begin{eqnarray*}
\binom{n}{d + 2 + L} (d + 2 + L)! (n^{-\alpha})^{(d + 2)(d + 1)/2 + L(d + 1)} &\leq& (d + 2 + L)! n^{d + 2 + L - \alpha((d + 2)(d + 1)/2 + L(d + 1))} \\
&=& (d + 2 + L)! n^{d + 2 - \alpha(d + 2)(d + 1)/2 + L(1 - \alpha(d + 1))}.
\end{eqnarray*}
The $(d + 2 + L)!$ factor comes from the number of ways to order the vertices, and the first inequality comes from the fact that $n^m \geq \binom{n}{m}$. Additionally, as there are only finitely many graphs on $d + 2 + L$ vertices, the probability that $\Delta$ contains a strongly connected, pure $(d+1)$-subcomplex on $d + 2 + L$ vertices is at most
\[C n^{d + 2 + \alpha(d + 2)(d + 1)/2 + L(1 - \alpha(d + 1))},\]
for some constant $C = C(d, L)$ depending on $d$ and $L$. Thus by setting $L$ large enough this can be made to be $o(n^{-M})$ for any constant $M$, as $1 - \alpha(d + 1) < 0$.

Next we turn our attention to small strongly connected $(d + 1)$-subcomplexes of $\Delta$ with minimum vertex degree $2d + 2$. Here ``small" means on fewer than $L$ vertices for $L$ a large constant. Note that the 1-skeleton of a strongly connected $(d + 1)$-subcomplex with minimum vertex degree $2d + 2$ has essential density at least $(d + 1)$, since each vertex contributes at least $2d + 2$ to the degree sum. Hence the probability that any such complex of $k$ vertices is contained in $\Delta(n, n^{-\alpha})$ is at most
\[\binom{n}{k}k!(n^{-\alpha})^{(d + 1)k} \leq \left(n^{1 - \alpha(d + 1)}\right)^k.\]
Now for any $k$ the number of graphs on $k$ vertices is at most $2^{k^2}$, so the expected number of pure, strongly connected $(d + 1)$-dimensional subcomplexes of $\Delta$ on at most $L$ vertices is at most
\begin{eqnarray*}
\sum_{k = 1}^{L} 2^{k^2} \left(n^{1 - \alpha(d + 1)}\right)^k &\leq& 2^{L^2} \sum_{k = 1}^{\infty} (n^{1 - \alpha(d + 1)}) = O(n^{1 - \alpha(d + 1)}.) 
\end{eqnarray*}
Thus given any $M$ there exists some large constant $L$ so that the probability that $\Delta \sim \Delta(n, n^{-\alpha})$ has a pure $(d + 1)$-dimensional, strongly connected subcomplex on more than $L$ vertices is at most $o(n^{-M})$. In this case we may bound the expected size of $\beta_{d + 1}(\Delta; \Q)$ by the trivial upper bound of $n^{d + 2}$. If $\Delta$ does not have a pure $(d + 1)$-dimensional, strongly connected subcomplex on $L$ vertices then the expected size of $\beta_{d + 1}(\Delta; \Q)$ is $O(n^{1 - \alpha(d + 1)}) = o(1).$ So for any $M > 0$ we have
\[ \E(\beta_{d + 1}(\Delta; \Q) \leq o(1) + n^{d + 2 - M} \]
which is seen to be $o(1)$ for $M \geq d + 2$.
\end{proof}

\section{Projective dimension, depth, and extremal Betti numbers}\label{sec:proj}

In this section we prove Theorem \ref{PdimTheorem}, which says that for $p = n^{-\alpha}$ with $1/(d+1) < \alpha < 1/d$ and $G \sim G(n,p)$, the projective dimension of $R/I_G$ over any field $\kk$ is bounded by
\[n-(d+1) \leq \pdim(R/I_G) \leq n-(d/2),\]
and for the case $\kk = {\mathbb Q}$ we have the precise value
\[\pdim(R/I_G) = n-(d+1).\]

Before turning to the proof we make some preliminary observations. As $R/I_G$ is a finitely generated module over $R = \kk[x_1, x_2, \dots, x_n]$, the Hilbert syzygy theorem says that $\pdim(R/I_G) \leq n$.  As for a lower bound, recall from Corollary \ref{cor:Krulldim} that if $p = n^{-\alpha}$ with $1/(d+1) < \alpha < 1/d$ then with high probability the Krull dimension of $R/I_G$ is at most $2d+2$.  Since the projective dimension of $R/I_G$ is at least the codimension of $I$ this gives a trivial lower bound of $\pdim(R/I_G) \geq n - (2d+2)$.

We claim that for $G \sim G(n, n^{-\alpha})$ an application of Hochster's formula and a result from \cite{Kah2014} (see Theorem \ref{Kah2014} above) provides the improved lower bound
\[\pdim(R/I_G) \geq n - (d + 1),\]
for coefficients in any field.  To see this we assume $G$ is a graph on $n$ vertices and apply Hochster's formula with $k = (d + 1)$ and $i = n - (d + 1)$. For any field $\kk$ we then get
\[\beta_{n - (d + 1), n}(R/I_G) = \dim_{\kk}(\tilde{H}_{d}(\Delta(G); \kk)).\]
\noindent
So if $G \sim G(n, n^{-\alpha})$ for $1/(d + 1) < \alpha < 1/d$ then Theorem \ref{Kah2014} implies that $H_d(\Delta(G); \Q) \neq 0$ with high probability, and hence by the universal coefficient  $H_d(\Delta(G), \kk)$ is nonvanishing for any field $\kk$.  To summarize our discussion we have the following warmup result.

\begin{proposition}\label{prop:pdim}
Fix an integer $d$ and suppose $p = n^{-\alpha}$ for $\alpha$ satisfying $1/(d+1) < \alpha < 1/d$.  Then for any coefficient field $\kk$ and $G \sim G(n,p)$ we have with high probability
\[n-(d+1) \leq \pdim(R/I_G) \leq n.\]
\end{proposition}

In particular we see that if $I_G$ is a random coedge ideal of fixed Krull dimension then the projective dimension of $R/I_G$ grows with the number of variables.  In what follows we will find better upper bounds on the projective dimension for arbitrary $\kk$, and the precise value for the case $\kk = \Q$. We will see that the techniques we use for the latter case are only applicable in the case of rational coefficients, and it remains an open problem to determine if the bound for $\Q$ holds for an arbitrary field. For our proofs we will use the combinatorial description of projective dimension via cohomological vertex connectivity $\kappa_{\kk}^i$ established in Lemma \ref{combinatorialpdim}.

\subsection{Spectral gaps and projective dimension over ${\mathbb Q}$}\label{sec:spectral}

We first consider the case that $\kk = {\mathbb Q}$ and examine the parameters $\kappa_{\Q}^i(G)$ for $i \geq 0$ and $G \sim G(n, p)$. It would be interesting to find closer estimates for $\kappa_{\Q}^i$ (see our discussion in Section \ref{sec:higherconn}), but for us it will be sufficient to establish the following result.

\begin{theorem}\label{thm:connectivity}
Suppose $d$ is an integer, fix $\alpha$ with $1/(d + 1) < \alpha < 1/d$, and let $G \sim G(n,n^{-\alpha})$. Then for any $0 \leq i < d$ and $L > 0$ we have  
\[\kappa_{\Q}^i(G) > L\]
with high probability.
\end{theorem}

Indeed this rough estimate on $\kappa_{\mathbb Q}^i(G)$ is enough to provide a proof of Theorem \ref{PdimTheorem} as follows.

\begin{proof}[Proof Theorem \ref{PdimTheorem} Part (2)]
Recall that by Lemma \ref{combinatorialpdim} we have for any $G$ that the projective dimension of $R/I_G$ over $\Q$ is given by
\[\pdim(R/I_G) = n - \min_{i \geq 1}\{i + \kappa_{\Q}^{i - 1}(G)\}. \]
First note that for $i \geq d + 2$, we have  $i + \kappa_{\Q}^{i - 1}(G) \geq d + 2$ and hence we can restrict to $i \leq d+1$ in the calculation of the quantity above. Now suppose $G \sim G(n,p)$ for $p = n^{-\alpha}$ as above.  By Theorem \ref{Kah2014}, we have $\kappa_{\mathbb{Q}}^d(G) = 0$ with high probability and hence for $i = d + 1$ we get that $i + \kappa_{\Q}^{i - 1}(G) = d + 1$. But for any $i \leq d$ we have by Theorem \ref{thm:connectivity}, that the value of $i + \kappa_{\Q}^{i - 1}(G)$ can be made arbitrarily large as $n \rightarrow \infty$. Therefore, for $n$ large enough, the minimum of the given set is $d + 1$, and we conclude that $\pdim(R/I_G) = n-(d+1)$ for the case $\kk = \Q$.
\end{proof}

The rest of this section will be devoted to the proof of Theorem \ref{thm:connectivity}.  We remark that the case $i=0$ follows from the Erd\H{o}s and R\'{e}nyi result regarding vertex connectivity of $G(n, p)$, discussed in Section \ref{sec:background}.  Here we will focus our attention on $i \geq 1$, where we will need some tools from spectral graph theory. Recall that by Theorem \ref{Kah2014} our choice of $\alpha$ satisfying $1/(d+1) < \alpha < 1/d$ implies that $\Delta \sim \Delta(n, n^{-\alpha})$ has rational homology only in degree $d$. Kahle's proof of this result (from \cite{Kah2014}) relies on \emph{Garland's method}, a tool from spectral graph theory that can be used to prove homology-vanishing statements for simplicial complexes.

We review the basics of spectral graph theory here and refer to \cite{Chu} for further details. If $G$ is a graph on $n$ vertices, denote by $A$ its $n \times n$ adjacency matrix and $D$ its degree matrix, i.e. the $n \times n$ diagonal matrix whose diagonal entries are the degree sequence of $G$. If $G$ has no isolated vertices, the \defi{normalized Laplacian} of $G$ is the $n \times n$ symmetric matrix given by $\mathcal{L} = I - D^{-1/2}AD^{-1/2}$, where $I$ is the $n \times n$ identity matrix. The \defi{eigenvalues} of $G$ are defined to be the eigenvalues of $\mathcal{L}$. It is well known that the eigenvalues of $G$ always lie in the interval $[0, 2]$, with the multiplicity of the eigenvalue 0 being the number of connected components of $G$ (in particular 0 is always the smallest eigenvalue of $G$). If we order the eigenvalues $0 = \lambda_1 \leq \lambda_2 \leq \cdots \leq \lambda_n$ then the \defi{spectral gap} of $G$ is defined as $\lambda_2$. There is a long history of relating the spectral gap of a graph $G$ to various graph properties (e.g. connectivity parameters) of $G$ itself. See for example \cite{Chu} for further discussion.

Also recall that if $\sigma$ is a face of a simplicial complex $X$ then the \defi{link} of $\sigma$ is given by
\[\lk_X(\sigma) = \{\tau \in X:\text{$\sigma \cup \tau \in X$ and $\sigma \cap \tau = \emptyset$}\}.\]
\noindent
Observe that if $X$ is a pure $k$-dimensional complex then the link of a $(k - 2)$-dimensional face is a graph.  With this in place we can state Garland's result \cite{Garland}, in particular a special case of it from \cite[Theorem 2.1]{BalSwi}

\begin{theorem}[Garland, Ballmann--\'{S}wi\c{a}tkowski]
Let $X$ be a pure $k$-dimensional finite simplicial complex such that for every $(k - 2)$-dimensional face $\sigma$ the link $\lk_X(\sigma)$ is connected and has spectral gap larger than $1 - 1/k$. Then $H^{k - 1}(X; \Q) = 0$.
\end{theorem}

Note that this special case of Garland's method gives a \emph{deterministic, combinatorial} condition to decide if rational homology is vanishing in some dimension. In his proof that $\Delta \sim \Delta(n, n^{-\alpha})$ has no homology below degree $d$, Kahle first shows that for any $i < d$, the $(i + 1)$-skeleton of $\Delta$ is pure dimensional. This means that every face of dimension at most $(i + 1)$ is contained in an $(i + 1)$-dimensional face. Next he uses the following result of Hoffman, Kahle, and Paquette to show that the link of every $(i - 1)$-dimensional face in the $(i + 1)$-skeleton has large spectral gap.

\begin{theorem}[Hoffman--Kahle--Paquette \cite{HKP}]\label{HKPspectralgap}
Let $G \sim G(n, p)$ be an Erd\H{o}s--R\'enyi graph. Then for any fixed $M \geq 0$ there is a constant $C := C(M)$ so that if 
\[p \geq \frac{(M + 1) \log n + C \sqrt{\log n} \log \log n}{n}\]
then $G$ is connected and 
$\lambda_2(G) \geq 1 - o(1)$ with probability $1 - o(n^{-M})$.
\end{theorem}

Our proof of Theorem \ref{thm:connectivity} proceeds similarly to Kahle's approach. Working at a logarithmic scale $p = n^{-\alpha}$, rather than trying to establish a sharp threshold, makes the argument shorter here. We begin with the following lemma that can be seen as a corollary to Theorem \ref{HKPspectralgap} for a two parameter random graph model. Here we follow notation for \emph{multiparameter random complexes} first introduced in  \cite{CosFar}, and define the \defi{two parameter Erd\H{o}s--R\'{e}nyi random graph model} $G(n; p_0, p_1)$ to be the model for random graphs where we start with a ground set of $n$ vertices, keep each vertex independently with probability $p_0$, and add an edge between each pair of surviving vertices independently with probability $p_1$.

\begin{lemma}\label{twoparameter}
Let $G \sim G(n; p_0, p_1)$ be a two parameter Erd\H{o}s--R\'{e}nyi random graph with the property that $p_0n \rightarrow \infty$ as $n \rightarrow \infty$. Furthermore assume that $p_0$ and $p_1$ satisfy
\[p_0p_1 = \omega \left( \frac{\log n}{n} \right).\]
Then for any $M \geq 0$ and $\epsilon > 0$ we have the following.
\begin{itemize}
    
    \item For $n$ large enough $\lambda_2(G) \geq 1 - \epsilon$ with probability at least $1 - (np_0)^{-M}$.
    \item In particular $G$ is connected with high probability.
    \end{itemize}
\end{lemma}

\begin{proof}
For any values of $n$, $p_0$, and $p_1$ the number of vertices of $G \sim G(n; p_0, p_1)$ is binomially distributed with $n$ trials and success probability $p_0$. Now if $G \sim G(n, p)$ has $k$ vertices then the edge set of $G$ is distributed as $G(k; 1, p_1) = G(k, p_1)$. In the average case $k = p_0 n$ and we can apply Theorem \ref{HKPspectralgap}. We make this argument precise. 

By conditioning on the size of $V(G)$, the probability that $G \sim G(n; p_0, p_1)$ has spectral gap smaller than $1 - \epsilon$ is at most 
\begingroup
\footnotesize
\[
\Pr(|V(G)| \geq 2np_0) + \Pr(|V(G)| \leq (1/2)np_0) + \sum_{k = (1/2)np_0}^{2np_0} \Pr\left(G \sim G(k, p_1) \text{ has } \lambda(G) < 1 - \epsilon \right)\left(\Pr(|V(G)| = k) \right).
\]
\endgroup

For $k \in [(1/2)np_0, 2np_0]$ and $p_1p_0 = \omega\left(\frac{\log n}{n} \right)$ we have that $p_1 = \omega \left(\frac{\log k}{k}\right)$. Hence for sufficiently large $n$ we have by Theorem \ref{HKPspectralgap} that the probability that $G \sim G(k, p_1)$ has spectral gap less than $1 - \epsilon$ is at most $(p_0n)^{-M - 2}$. It follows that for $n$ large enough we have
\begin{eqnarray*}
\sum_{k = (1/2)np_0}^{2np_0} \Pr\left(G \sim G(k, p_1) \text{ has } \lambda(G) < 1 - \epsilon \right)\left(\Pr(|V(G)| = k) \right) \leq 2np_0 (p_0n)^{-M - 2} = 2(p_0n)^{-M - 1}.
\end{eqnarray*}
Additionally by Chernoff bound since $V(G)$ is binomially distributed we have
\begin{eqnarray*}
\Pr(|V(G)| \geq 2np_0) + \Pr(|V(G)| \leq ((1/2)np_0)) \leq \exp(-np_0/10).
\end{eqnarray*}
Therefore for $n$ large enough the probability that $G \sim G(n; p_0, p_1)$ has spectral gap smaller than $1 - \epsilon$ for $p_0p_1 = \omega\left(\frac{\log n}{n} \right)$ is at most $(np_0)^{-M}$.

\end{proof}

\begin{proof}[Proof of Theorem \ref{thm:connectivity}]
Fix $\alpha$ with $1/(d + 1) < \alpha < 1/d$ and let $L > 0$. We show that if $G \sim G(n,n^{-{\alpha}})$ and $i < d$ then with high probability we have $\kappa_{\Q}^{i}(G) > L$. For this we use a first moment argument to bound the probability that there exists a set of $L$ vertices with the property that the subcomplex of $\Delta \sim \Delta(n, n^{-\alpha})$ obtained by deleting those $L$ vertices has homology in degree $i$. Consider a fixed set $\{x_1, \dots, x_L\}$ of vertices from $[n]$. The random complex induced on the remaining vertices is distributed as $\Delta(n - L, n^{-\alpha})$. We claim that with probability at most $o(n^{-L})$ such a complex has no homology in degree $i$.  For this we use Garland's method, and in particular we prove that the following two conditions occur with large enough probability:
\begin{enumerate}
    \item $\Delta \sim \Delta(n - L, n^{-\alpha})$ has pure-dimensional $(i + 1)$-skeleton.
    \item Every $(i - 1)$-dimensional face $\sigma$ of $\Delta \sim \Delta(n - L, n^{-\alpha})$ has that the graph of its link has spectral gap larger than $1 - \frac{1}{i + 1}$.
\end{enumerate}
Of the two, Condition (1) is easier to verify. For this we must check that for $j \leq i + 1$ every $j$-clique of $G \sim G(n - L, n^{-\alpha})$ is contained in an $(i + 2)$-clique. For a fixed set of $j \leq i + 1$ vertices, the probability that there is a clique induced on those vertices that is not contained in a $(j + 1)$-clique is at most
\begin{eqnarray*}
(n^{-\alpha})^{\binom{j}{2}}(1 - n^{-j\alpha})^{n - L - j} \leq n^{-\binom{j}{2} \alpha} \exp(- \Omega(n^{1 - j\alpha})).
\end{eqnarray*}
Therefore by applying a union bound the probability that there is a $j$-clique not contained in $(j + 1)$-clique is at most 
\[n^{j - \binom{j}{2} \alpha} \exp(-\Omega(n^{1 - j\alpha})).\]
Hence with probability at most
\[\sum_{j = 0}^{i + 1} n^{j - \binom{j}{2}\alpha} \exp(-\Omega(n^{1 - j\alpha})) \]
the $(i + 1)$-skeleton of $\Delta \sim \Delta(n - L, n^{-\alpha})$ is not pure $(i + 1)$-dimensional. This probability is exponentially small in $n$ when $\alpha < 1/(i + 1)$, which holds since $i + 1 \leq d$ and $\alpha < 1/d$.

Now we turn our attention to Condition (2). Given a set $\{v_1, ..., v_i\}$ of $i$ vertices in $[n] \setminus \{x_1, ..., x_L\}$, we have that 
\[\lk(v_1) \cap \cdots \cap \lk(v_i)\]
is distributed as $G(n - L - i; n^{-i \alpha}, n^{-\alpha})$.  We then have 
\[n^{-i \alpha}n^{-\alpha} = n^{-\alpha(i + 1)} = \omega\left(\frac{\log n}{n}\right)\]
since $i + 1 \leq 1/d$ and $\alpha < 1/d$. Therefore for every $M > 0$ we have for $n$ large enough
\[\Pr\left(\lambda_2(\lk(v_1) \cap \cdots \cap \lk(v_i)) < 1 - \frac{1}{i + 1}\right) = (n^{1 - i \alpha})^{-M}.\]
This means for a given set of $L$ vertices removed from $[n]$, and a given $(i - 1)$-dimensional face of $\Delta \sim \Delta(n - L, n^{-\alpha})$, the probability that the graph of its link has spectral gap smaller than $1 - \frac{1}{i + 1}$ is at most $(n^{1 - i \alpha})^{-M}$. 

Thus summing over all $\binom{n}{L}$ choices for the vertices to delete and all at most $\binom{n}{i}$ possible $(i - 1)$-dimensional faces, the probability that the large spectral gap condition fails for some set of $L$ deleted vertices and some $(i - 1)$-dimensional face is at most
$n^{L + i - M(1 - i)\alpha}$, which is $o(1)$ if we pick $M$ to be a sufficiently large constant. Thus the probability that Condition (2) fails is $o(1)$, completing the proof that rational homology in degree $i$ vanishes.
\end{proof}
As a corollary to our proof of Theorem \ref{thm:connectivity} we have the following result which will be convenient for the proof of Theorem \ref{corstaircase}. This result essentially follows from Kahle's proof of Theorem \ref{Kah2014}, but does not appear exactly in this form in \cite{Kah2014}.
\begin{corollary}\label{lemma2ndmomentabove}
Suppose $\alpha < 1/d$ and let $M > 0$. Then for $\Delta \sim \Delta(n,n^{-\alpha})$,
\[\Pr(\beta_{d - 1}(\Delta; \Q) = 0) \geq 1 - o(n^{-M}),\] and hence  \[\mathbb{E}(\beta_{d - 1}(\Delta; \Q)) = o(1).\]
\end{corollary}

\subsection{Projective dimension over $\kk$}
We next turn our attention to the projective dimension of $R/I_G$ over an arbitrary field, where new arguments are needed. The bouquet of spheres conjecture (Conjecture \ref{conj:spheres}) implies that the integral $i$th homology of $\Delta(n,p)$ is vanishing when $p = n^{-\alpha}$ and $\alpha < 1/(i + 1)$.  As far as we know, the best known vanishing result for $\tilde H_i(\Delta(n, p))$ with integer coefficients is an earlier result of Kahle from \cite{Kah}, which says that if $p = n^{-\alpha}$ and $\alpha < 1/(2i + 1)$ then $\tilde H_i(\Delta(n,p)) = 0$ with high probability. In fact Kahle's result says that if $\alpha < 1/(2i + 1)$, then $\Delta \sim \Delta(n, n^{-\alpha})$ is topologically $i$-connected, that is $\pi_j(\Delta) = 0$ for all $j \leq i$. With this in mind, for a graph $G$ we let $\hat{\kappa}^i(G)$ denote the minimum size of a set $S \subseteq G$ whose removal destroys this property, so that
\[
\hat{\kappa}^i(G) := \min_{S \subseteq G}\left\{|S| : \Delta(V(G) \setminus S) \text{ is not topologically $i$-connected}\right\}.
\]
We then have that for every field $\kk$ and every $i \geq 0$, 
\[\hat{\kappa}^i(G) \leq \kappa_{\kk}^i(G).\]
Indeed for any integer $t$ if $\hat{\kappa}^i(G) > t$, then for every set $S$ of size $t$, the flag complex over $G[V(G) \setminus S]$ is $i$-connected. It follows therefore by the universal coefficient theorem that the $i$th cohomology group over $\kk$ of any such complex is zero.

To establish the general bounds for $\pdim(R/I_G)$ over any field $\kk$ in Theorem \ref{PdimTheorem}, it suffices to show that for $i \leq \frac{d - 1}{2}$ we have $\hat{\kappa}^i(G) = \omega(1)$ for $G \sim G(n, n^{-\alpha})$ and $1/(d + 1) < \alpha < 1/d$. Our proof essentially follows the proof of Kahle from \cite{Kah}, although we require sharper bounds on the relevant bad events.  We prove the following lemma.
\begin{lemma}\label{lemmma:iconnected}
Suppose $d \geq 1$ is fixed, let $\alpha < 1/d$.  If $p > n^{-\alpha}$ then for any $M \geq 0$ we have with probability at least $1 - o(n^{-M})$ that $\Delta \sim \Delta(n, p)$ is topologically $\lfloor \frac{d - 1}{2} \rfloor$-connected. 
\end{lemma}

The proof of this lemma essentially follows the proof of Theorem 3.4 from \cite{Kah}, and in particular we use the following key \emph{deterministic} lemma of \cite{Kah}. In what follows we use $\st(v)$ to denote the star of $v$, the subcomplex of $\Delta(G)$ generated by the set of faces that contain $v$.
\begin{theorem}\label{commonneighborlemma}\cite[Lemma 4.2]{Kah}
Let $k \geq 1$ and suppose that $G$ is a graph where every set of $2k + 1$ vertices have a common neighbor, and for every set of $\ell \leq 2k$ vertices $v_1, .., v_{\ell}$, the intersection 
\[\st(v_1) \cap \cdots \cap \st(v_{\ell})\]
is path connected. Then $\Delta(G)$ is topologically $k$-connected. 
\end{theorem}

\begin{proof}[Proof of Lemma \ref{lemmma:iconnected}]
We bound the probability that the assumptions of Theorem \ref{commonneighborlemma} fail to hold. First for $k = \lfloor \frac{d - 1}{2} \rfloor$ we bound the probability that there is a set of $2k + 1$ vertices with no common neighbor. Observe that $2k + 1 \leq d$, and so it suffices to show that for $\alpha < 1/d$ and $p > n^{-\alpha}$, every set of $d$ vertices in $G \sim G(n, p)$ has a common neighbor with very high probability. For a fixed set $T$ of $d$ vertices, the probability that the vertices of $T$ have no common neighbor in $G(n, p)$ is 
\[(1 - p^d)^{n - d} \leq \exp(-\Omega(n^{1 - d\alpha})).\]
Hence the probability that there exists a set of $d$ vertices with no common neighbor is at most $n^d \exp(-\Omega(n^{1 - d\alpha})) = \exp(-\Omega(n^{1 - d\alpha}))$.

Next we bound the probability that a collection of $\ell \leq 2k$ vertices $v_1, ..., v_{\ell}$ has the property that $\bigcap_{i=1}^\ell \st(v_i) = \st(v_1) \cap \st(v_2) \cap \cdots \cap \st(v_{\ell})$ is disconnected. We first observe that if $\bigcap_{i=1}^\ell \st(v_i)$ is disconnected then $\bigcap_{i=1}^\ell \lk(v_i) = \lk(v_1) \cap \cdots \cap \lk(v_{\ell})$ is disconnected. If these two sets agree then this is clear. Otherwise each vertex in $\bigcap_{i=1}^\ell \st(v_i) \backslash \bigcap_{i=1}^\ell \lk(v_i)$ belongs to $\{v_1, ..., v_{\ell}\}$, and each of these vertices is adjacent to every vertex in $\bigcap_{i=1}^\ell \lk(v_i)$. Hence if $\bigcap_{i=1}^\ell \lk(v_i)$ is nonempty and $\bigcap_{i=1}^\ell \st(v_i) \neq \bigcap_{i=1}^\ell \lk(v_i)$ then $\bigcap_{i=1}^\ell \st(v_i)$ is connected. We have already shown that with very high probability $\bigcap_{i=1}^\ell \lk(v_i)$ is nonempty, and so it remains to show that with very high probability it is connected.


For this observe that the 1-skeleton of $\bigcap_{i=1}^\ell \lk(v_i)$, which is all we have to consider to show path connectivity, is a random graph on a ground set of $n - l$ vertices with each vertex included independently with probability $p^{\ell} > p^{2k} > n^{-2k\alpha} \geq n^{-(d - 1) \alpha}$, and each edge included independently with probability $p > n^{-\alpha}$ conditioned on both endpoints being in the graph. Therefore $\bigcap_{i=1}^\ell \lk(v_i)$ is a two parameter random graph with $p_0p_1 > n^{-d \alpha} = \omega\left(\frac{\log n }{n}\right)$ and $p_0n > n^{1 - (d - 1) \alpha}$, where the latter quantity tends to infinity with $n$ as $\alpha < 1/(d - 1)$. Thus Lemma \ref{twoparameter} applies and we conclude that for any constant $M$  the complex $\bigcap_{i=1}^\ell \lk(v_i)$ is connected with probability at least $1 - o(n^{-M})$.  The result then follows from Theorem \ref{commonneighborlemma}.

\end{proof}


\begin{proof}[Proof of Theorem \ref{PdimTheorem} Part (1)]
Suppose $p = n^{-\alpha}$ with $1/(d+1) < \alpha < 1/d$.  By Proposition \ref{prop:pdim} we have that $\pdim(R/I_G) \geq n-(d+1)$ with high probability for any choice of coefficients $\kk$.  For the other inequality note that Lemma \ref{lemmma:iconnected}, along with Theorem \ref{commonneighborlemma}, implies that $\kappa_{\kk}^i(G) \geq \hat{\kappa}^i(G) = \omega(1)$ for $i \leq \frac{d - 1}{2}$. Hence by Lemma \ref{combinatorialpdim} we have that $\pdim(R/I_G) \leq n-(d/2)$ with high probability.
\end{proof}

\subsection{Depth, Cohen--Macaulay properties, and extremal Betti numbers}
Next we establish some further algebraic corollaries of our results. Recall that the \emph{depth} of the module $R/I_G$ is given by the maximum length of a regular sequence in $R/I_G$, and can also be defined by the Auslander-Buchsbaum formula as
\[\pdim(R/I_G) + \depth(R/I_G) = n. \]

Hence as an immediate corollary to Theorem \ref{PdimTheorem} we get the following.

\begin{corollary}\label{cor:depth}
Suppose $p = n^{-\alpha}$ with $1/(d+1) < \alpha < 1/d$ for some $d \geq 1$ and let $G \sim G(n,p)$. Then for any coefficient field $\kk$ with high probability the $R$-module $R/I_G$ satisfies
 \[d/2 \leq  \depth(R/I_G) \leq d+1.\]
For the case $\kk = {\mathbb Q}$ we have with high probability that
 \[\depth(R/I_G) = d+1.\]
\end{corollary}
 
 This leads us to a good description of the (lack of) Cohen--Macaulay properties of $R/I_G$ and a proof of Corollary \ref{cor:depthCM}. For this recall that an $R$-module $M$ is \emph{Cohen--Macaulay} (over the field ${\mathbb K}$) if $\dim(M) = \depth(M)$. 
 
\begin{proof}[Proof of Corollary \ref{cor:depthCM}]
As usual assume $p = n^{-\alpha}$ for $1/(d+1) < \alpha < 1/d$ and let $G \sim G(n,p)$. From Corollary \ref{cor:depth} we see that $\depth(R/I_G) \leq d+1$ for any field of coefficients, with equality in the case that $\kk = {\mathbb Q}$. From Corollary \ref{cor:Krulldim} we know that with high probability the Krull dimension of $R/I_G$ is at least $2d+1$, which proves the first part of the statement.  Furthermore, since $d \geq 1$ we have that $\depth(R/I)$ is less than the Krull dimension of $R/I$ with high probability. From this it follows that $R/I_G$ is not Cohen--Macaulay over any field $\kk$.
\end{proof}

This complements results from \cite{ErmYan} where the question of Cohen--Macaulay properties was considered for the regime $1/n^{2/3} < p < \large(\frac{\log n}{n}\large)^{2/k+3}$. 

Our results also lead to an understanding of the extremal Betti numbers of random quadratic ideals.  Recall that a nonzero graded Betti number $\beta_{i,i+k}(R/I)$ is said to be \emph{extremal} if $\beta_{i',i'+k'}(R/I) = 0$ for all pairs $(i',k') \neq (i,k)$ with $i' \geq i$ and $k' \geq k$. In \cite{HKP} it is shown that for any integer $r$ and $b$ with $1 \leq b \leq r$ there exists a graph $G$ for which $\reg(R/I_G) = r$ and such that $R/I_G$ has $b$ extremal Betti numbers.   As a corollary of our results we see that for the case of random coedge ideals the extremal Betti numbers are very easy to describe.

\begin{proof}[Proof of Corollary \ref{cor:extremal}]
As above suppose $p = n^{-\alpha}$ with $1/(d+1) < \alpha < 1/d$.  For $G \sim G(n,p)$ we have seen that with high probability $\beta_{n - (d + 1), n}(R/I_G) \neq 0$, and in fact this must be an extremal Betti number by Hochster's formula.  From Theorem \ref{RegularityTheorem} we see that if $k \geq d+2$ then with high probability $\beta_{i,i+k} = 0$ and hence there can be no extremal Betti numbers in rows larger than $d+1$.  Since we are assuming $\kk = {\mathbb Q}$, part (2) of Theorem \ref{PdimTheorem} in particular implies that if $i > n-(d + 1)$ then $\beta_{i,j} = 0$ and hence there can also be no extremal Betti numbers in rows smaller than $d+1$.  The result follows.
\end{proof}

In \cite{BCP} it is shown that taking the revlex generic initial ideal preserves extremal Betti numbers. This implies that $I$ has a single extremal Betti number if and only if $\text{gin}(I)$ does.  From this we can conclude that if $R/I$ is Cohen--Macaulay then $R/I$ has an extremal number.   Our random coedge ideals provide examples of non Cohen--Macaulay rings $R/I_G$ that have a single extremal Betti number.

We can summarize our results in this section in terms of the Betti table of $R/I_G$.  As usual suppose $p=n^{-\alpha}$ for $1/(d + 1) < \alpha < 1/d$ and let $G \sim G(n,p)$.  Then asymptotically the right-hand side of the Betti table of $R/I_G$ is given by the diagram in Table \ref{tbl:rightside}.  For the case that $\kk = \Q$, we therefore have Corollary \ref{cor:extremal}.

\begin{table}[h]
    \centering
    \begin{tabular}{c|ccccccccc}
        $\cdot$ & $n - (d+1)$ & $n - d$ & $\cdots$ & $n - (d + 3)/2$ & $n - (d + 1)/2$ & $\cdots$ & $n - 2$ & $n - 1$ \\ \hline
        1 & {0} & {0} & $\cdots$ & {0} & {0} & $\cdots$ & {0} & {0}\\
        2 & {0} & {0} & $\cdots$ & {0} & {0} & $\cdots$ & {0} & -\\
        $\vdots$ & $\vdots$ & $\vdots$ & $\vdots$ & $\vdots$ & $\vdots$ & $\vdots$ & $\vdots$ & $\vdots$  \\
        $\frac{d + 1}{2}$ &{0} & {0} & $\cdots$ & {0} & {0} & $\cdots$ & - & -\\ 
        $\frac{d + 3}{2}$ & $0^{\dagger}$ & $0^{\dagger}$ & $\cdots$ & $0^{\dagger}$ & - & $\cdots$ & - & -\\ 
        $\vdots$ & $\vdots$ & $\vdots$ & $\vdots$ & $\vdots$ & $\vdots$ & $\vdots$ & $\vdots$ & $\vdots$ \\
        $d$ & $0^{\dagger}$ & $0^{\dagger}$ & $\cdots$ & - & - & $\cdots$ & - & -\\ 
        $d + 1$ & $*$ & - & $\cdots$ & - & - & $\cdots$ & - & -\\ 
        
    \end{tabular}
    \caption{The right-hand side of the Betti table of $R/I_G$, where $p=n^{-\alpha}$ with $1/(d - 1) < \alpha < 1/d$. A $0$ indicates that the entry has been shown to be zero over any field, and a $0^{\dagger}$ indicates that the Betti number is provably zero when $\kk =\Q$.  We use $*$ to indicate a nonzero Betti number, which is extremal in the case $\kk = \Q$.}
    \label{tbl:rightside}
\end{table}

\section{Distribution of Betti numbers}\label{sec:log}


In this section we study the distribution of nonzero entries in the Betti table of $R/I_G$ for $G \sim G(n,p)$ as above.  This is inspired by results of Erman and Yang \cite{ErmYan}, and in particular a theorem regarding the density of nonzero entries in the Betti table of $R/I_G$ for certain values of $p$. We discussed this result in Section \ref{sec:intro} but we provide the details here.  

For any $R$-module $M$ and integer $k \geq 1$, define the parameter $\rho_k(M)$ as
\[\rho_k(M) = \frac{|\{i\in [0, \pdim(M)] \text{ where $\beta_{i,i+k}(M) \neq 0$}\}|}{\pdim(M) + 1}.\]

Recall that $\pdim(M)$ denotes the projective dimension of $M$. If $\{I_n\}$ is a family of ideals where $\pdim(I_n) \rightarrow n$ and $k$ is some fixed integer, a motivating question in \cite{ErmYan} is to determine what conditions will guarantee that  $\rho_k(R/I_n) \rightarrow 1$ as $n \rightarrow \infty$. A main result of \cite{ErmYan} is the following answer for the case of coedge ideals in certain regimes for the parameter $p$.

\begin{theorem}[Theorem 1.3 of \cite{ErmYan}]\label{thm:density}
Fix some $d \geq 1$. Let $G \sim G(n,p)$ with $\frac{1}{n^{1/d}} \ll p \ll 1$. For each $1 \leq k \leq d+1$ we have 
\[\rho_k(R/I_G) \rightarrow 1\]
in probability.
\end{theorem}

As we have seen, for these values of $p$ we have that $\pdim(R/I_G) \rightarrow \infty$ as $n \rightarrow \infty$. Hence Theorem \ref{thm:density} says that for any $k \leq d+1$, as $n \rightarrow \infty$ the density of nonzero entries in the $k$th row of the Betti table of $R/I_G$ approaches one as $n \rightarrow \infty$.

We wish to have a better understanding of the magnitude of these nonzero Betti numbers. Our Theorem \ref{corstaircase} describes a dichotomy in the \emph{normalized} Betti numbers $\overline \beta_{i,j}(R/I_G) = \frac{\beta_{i,j}(R/I_G)}{\binom{n}{j}}$ in these rows, for the case of ${\mathbb K} = {\mathbb Q}$.  In particular we see that if $\alpha$ satisfies $1/(d+1) < \alpha < 1/d$ then for any $k \leq d+1$ the $k$th row of the normalized Betti table of $R/I_G$ will have an interval where $\overline{\beta}_{i,j} \rightarrow \infty$ if $i$ is in the interval, but where $\overline{\beta}_{i,i+k} \rightarrow 0$ otherwise.
  We refer to Figure \ref{summary} for an illustration of this phenomenon.  We next turn to the proof of Theorem \ref{corstaircase}.

\begin{figure}[h]
\centering
\includegraphics[width = 4 in]{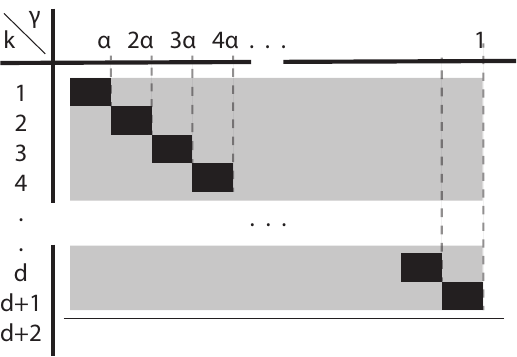}
\caption{The normalized Betti table of $R/I_G$ where $1/(d + 1) < \alpha < 1/d$ and $G \sim G(n, n^{-\alpha})$.  Here the columns $\gamma \in [0,1]$ are indexed on a logarithmic scale $\log_n$ and we see $\overline \beta_{i,i+k}$ with $i = n^\gamma$. Within the darkened band in each row the normalized Betti numbers tend to infinity, outside the darkened band they tend to zero (but are nonvanishing by \cite{ErmYan}). Beyond row $d+1$ all entries are zero (see Theorem \ref{RegularityTheorem}), }\label{summary}.
\end{figure}

\begin{proof}[Proof of Theorem \ref{corstaircase}]
Let $1/(d+1) < \alpha < 1/d$ and fix $k \leq d + 1$.  We first prove the second part of the statement regarding the vanishing of $\overline{\beta}_{i,i+k}$ when $i = n^\gamma - k$ with 
$\gamma \in (0, 1) \setminus [(k - 1)\alpha, k \alpha]$.  By Hochster's formula we have
\[\beta_{i, i + k}(R/I_G) = \sum_{S \in \binom{[n]}{n^{\gamma}}} \beta_{k - 1}(\Delta(S); \Q).\]
For a fixed set $S$ of $i + k$ vertices an induced subcomplex of $\Delta \sim \Delta(n, n^{-\alpha})$ on vertex set $S$ is distributed as $\Delta \sim \Delta(i + k, n^{-\alpha})$. As $i = n^{\gamma} - k$, this is distributed as $\Delta(i, i^{-\alpha/\gamma})$, up to the constant $k$ which doesn't matter in the limit. As $i \rightarrow \infty$ and $\alpha/\gamma$ is outside of the interval $(1/k, 1/(k - 1))$ we have $\E(\beta_{k - 1}(\Delta(S); \Q)) = o(1)$ for a fixed set $S$. For the case $\alpha/\gamma < 1/k$ this follows from Corollary \ref{lemma2ndmomentabove}, and for $\alpha/\gamma > 1/(k - 1)$ from Lemma \ref{lemma2ndmomentbelow}. 

Thus by linearity of expectation and Hochster's formula, we have $\E(\beta_{i, i + k}(R/I_G)) = o(\binom{n}{i + k})$. From this it follows that \[\frac{\beta_{i, i + k}(R/I_G)}{\binom{n}{i + k}} \rightarrow 0,\]
with probability tending to one by Markov's inequality.

Next we consider $\overline{\beta}_{i,i+k}$ with $i=n^\gamma - k$ for the case $\gamma \in ((k - 1) \alpha, k \alpha)$. By essentially the same argument as above we can show that for a fixed set $S$ of size $n^{\gamma} - k$, we have  by Theorem \ref{Kah2014} that $\beta_{k - 1}(\Delta(S); \Q) \neq 0$. However, we need to be a bit more careful here to actually show that $\beta_{i, i + k}(R/I_G) = \omega(\binom{n}{i + k})$, which we require to conclude that $\overline{\beta}_{i,i+k}(R/I_G) \rightarrow \infty$.  In particular we need to say something about the magnitude of the (topological) Betti number $\beta_{k - 1}(\Delta(S); \Q)$. Fortunately, while estimates of this type are not in the statement of Theorem \ref{Kah2014}, they are in the proof that Kahle provides in \cite{Kah2014}. For completeness we include an argument estimating $\beta_{k - 1}(\Delta(S); \Q)$ here.

Recall that if $\Delta$ is a simplicial complex its \defi{$f$-vector} $(f_0, f_1, \dots )$ has entries $f_i$ given by the number of $i$-dimensional faces of $\Delta$.  By the definition of simplicial homology we can bound $\beta_d = \beta_d(\Delta, \Q)$ by
\[f_{d} - f_{d - 1} - f_{d + 1} \leq \beta_d \leq f_d.\]
\noindent
For $\Delta \sim \Delta(n, n^{-\alpha})$, recall that $f_m(\Delta)$ is the number of $(m + 1)$-cliques in $G \sim G(n, n^{-\alpha})$. A graph $H$ is said to be \defi{strictly balanced} if $\nu(H)$ is attained only by the graph $H$ itself and not by any of its proper subgraphs (see Section \ref{sec:randomgraphs} for a definition of $\nu(H)$). It is easy to see that $K_{m + 1}$ is strictly balanced with density $\nu(K_{m+1}) = m/2$. Therefore from \cite[Theorem 4.4.4]{AS} we have with high probability that the number of $m + 1$ cliques in $G \sim G(n, n^{-\alpha})$ is $\Theta(n^{m + 1 - \alpha \binom{m + 1}{2}})$. Combining this with the bounds on $\beta_d$ described above we have that with high probability $\beta_d(\Delta) = \Theta(n^{d + 1 - \alpha\binom{d + 1}{2}})$ when $\alpha < 1/d$ and $n^{d + 1 - \alpha\binom{d + 1}{2}}$ tends to infinity with $n$ when $\alpha < 2/d$.

Now let $S$ be a fixed set of $i + k = n^{\gamma}$ vertices from $[n]$. Then $\Delta(S) \sim \Delta(i, i^{-\alpha/\gamma})$, again ignoring the negligible constant $k$, when $G \sim G(n,n^{-\alpha})$. Since $\gamma \in (0, 1) \cap ((k - 1)\alpha, k \alpha)$, $1/k < \alpha/\gamma < 1/(k-1)$ and $i$ tends to infinity with $n$, we have that from the above discussion that $\beta_{k - 1}(S) = \omega(1)$ with high probability. From here we can show that $\overline{\beta}_{i, i + k}(R/I_G)$ diverges to infinity in probability.

Suppose for contradiction that there are constants $L$ and $\epsilon$ so that for all $n$ large enough, for $G \sim G(n, n^{-\alpha})$
\[\Pr(\overline{\beta}_{i, i + k}(R/I_G) \geq L) \leq 1 - \epsilon.\]
Then if we generate $G \sim G(n, n^{-\alpha})$ and choose a random subset $S$ of $i + k$ vertices and compute $\beta_{k - 1}(\Delta(S))$ we have
\[\Pr\left(\beta_{k - 1}(\Delta(S)) \geq \frac{2L}{\epsilon}\right) \leq 1 - \epsilon/2.\]
Indeed if $G$ is such that $\overline{\beta}_{i, i + k}(R/I_G) \leq L$ then by Markov's inequality and Hochster's formula the probability that $\beta_{k - 1}(\Delta(S)) \geq \frac{2L}{\epsilon}$ for $S$ a randomly selected set of $i + k$ vertices is at most $\frac{\epsilon}{2}$. Moreover the probability that $G$ has $\overline{\beta}_{i, i + k}(R/I_G) \geq L$ is at most $1 - \epsilon$. But $\Delta(S)$ selected by generating $G \sim G(n, n^{-\alpha})$ and then picking $S$ uniformly at random is distributed as $\Delta(i, i^{-\alpha/\gamma})$, and we know that $\beta_{k - 1}(\Delta(S)) = \omega(1)$ with high probability in this case, so we have a contradiction.


\end{proof}

\section{Further thoughts and open questions}\label{sec:further}

\subsection{Higher degree ideals}
Having studied algebraic properties of random quadratic monomial ideals, a natural question to ask is whether there is an analogous construction for equigenerated monomial ideals in higher fixed degree.  Indeed, one can see that a monomial ideal of degree $d$ corresponds to the Stanley--Reisner ideal of a simplicial complex $\Delta$ that has a complete $(d-1)$-skeleton with no empty simplices of dimension $d+1$ or larger, so that whenever the $d$-skeleton of a simplex appears in $\Delta$ we have that the simplex itself is contained in $\Delta$.

In fact there is already a well studied model for such complexes, given by the \defi{multiparameter model} for random simplicial complexes $X(n; p_0 p_1, p_2, p_3, ...)$, introduced by Costa and Farber in \cite{CosFar}. To sample from this model one starts with $n$ vertices as the ground set and includes each $k$-simplex independently with probability $p_k$, provided that all of its boundary facets are included. That is, each vertex is included independently with probability $p_0$, each edge between existing vertices is included independently with probability $p_1$, each triangle from the resulting graph is filled in independently with a 2-dimensional face with probability $p_2$, and so on. The model $\Delta(n, p)$ is exactly $X(n; 1, p, 1, 1, 1...)$.

With this in mind, for fixed $m \in \N$ define $\Delta_m(n, p)$ to be the multiparameter model with $p_m = p$ and $p_i = 1$ for $i \neq m$ to obtain a random complex in which every minimal excluded face is $m$-dimensional. Note that $\Delta_m(n, p)$ is the Linial--Meshulam model $Y_m(n, p)$, first introduced in \cite{LinMes}, together with ``filling in" all empty simplices of dimension larger than $d$.
The Stanley--Reisner ideal $I_\Delta$ of a complex $\Delta \sim\Delta_m(n, p)$ is generated by squarefree monomials of degree $m + 1$, and using Hochster's formula one could establish results regarding $\reg(R/I_\Delta)$ and $\pdim(R/I_\Delta)$. Based on results established for the multiparameter model in \cite{Fow} it seems reasonable to believe that $R/I_{\Delta}$ for $\Delta \sim \Delta_m(n, p)$ behaves similarly to $R/I_G$ for $G \sim G(n, p)$ in terms of having one extremal Betti number. Adapting standard methods from the literature on random simplicial complexes we suspect the following is true.
\begin{conjecture}
For $d \geq m \geq 1$ fixed suppose $\alpha$ 
satisfies \[\frac{1}{\binom{d + 1}{m}} < \alpha < \frac{1}{\binom{d}{m}},\]
and let $\Delta \sim \Delta_m(n, n^{-\alpha})$. Then with high probability he have that $\Q[x_1, ..., x_n]/I_{\Delta}$ has regularity and depth both equal to $d + 1$, and has one extremal Betti number given by $\beta_{n - (d + 1), n}$.
\end{conjecture}

\subsection{Other probability regimes}
In this paper our primary method for generating random monomial ideals involved the Erd\H{o}s--R\'enyi model of random graphs $G(n,p)$ where $p = n^{-\alpha}$ with $1/(d+1) < \alpha < 1/d$.  As we discussed earlier, one reason for this is that the clique number of $G \sim G(n, p)$ in this regime is bounded in terms of $\alpha$ so the resulting ring $R/I_G$ has bounded Krull dimension.

On the other hand one could also consider $R/I_G$ for other values of $p$. We note that for $p$ itself a constant it is not hard to see that the Krull dimension of $R/I_G$ is order $\Theta(\log n)$ and the regularity can be bounded below by $\Omega(\log n)$. Interestingly however, Theorem 3.1.1 of Greg Malen's Ph.D. thesis \cite{Malen} proves that for $p$ constant, integer homology of $\Delta \sim \Delta(n, p)$ is nonvanishing in more than one dimension. In particular, he shows that integer homology is nonvanishing in $\Theta(\log \log n)$ many dimensions. This implies that, unlike the sparse regime we consider here, $R/I_G$ will have more than one extremal Betti number when $G \sim G(n, p)$ and $p$ is a constant. In addition, the regime $p = 1 - \lambda/n$ for $\lambda$ constant is considered in \cite{BanYog}, where they establish results about the regularity of $R/I_G$ and prove it is order $\Theta(n)$. 

One question in particular that would be interesting to consider is whether or not there exists a choice of $p$ so that with high probability $G \sim G(n, p)$ has regularity strictly larger than the bound given by the maximum dimension of nonvanishing homology of $\Delta(G)$.  In \cite{EngOrl} Engstr\"om and Orlich study edge ideals of random \emph{unlabeled} graphs and establish results regarding their regularity in terms of vanishing of what they call \emph{parabolic} Betti numbers.



\subsection{Thresholds for higher $\kappa$}\label{sec:higherconn}
In Section \ref{sec:proj} we saw that the projective dimension of $R/I_G$ for a random coedge ideal $I_G$ can be understood in terms of the parameters $\kappa_{\kk}^i(G)$.  In particular the index of the last nonzero entry in row $i$ of the Betti table of $R/I_G$ is given by $n - i - \kappa_{\kk}^{i-1}(G)$.

For the case of $i=0$ (corresponding to the first row of the Betti table) this data is well understood in terms of well known thresholds for the vertex $\ell$-connectivity of a random graph (see Section \ref{sec:randomgraphs}).  In particular the threshold for $\ell$-connectivity for a random graph $G \sim G(n,p)$ corresponds to the threshold for $G$ to have minimum degree $\ell$, which is clearly a necessary condition.

There is analogous question for the other strands involving $\kappa_{\kk}^i$. This was studied for instance by Babson and Welker in \cite{BabWel}, where they establish sufficient conditions on $p$ and $i$ to conclude that $\Delta \in \Delta(n, p)$ has $\kappa_{\mathbb{F}_2}^i$ equal to the minimum degree of an $i$-dimensional face. As far as we know questions of this type for rational coefficients, where Garland's method may be used, have not been considered. For instance one might restrict to $p = n^{-\alpha}$ for $1/(d+1) < \alpha < 1/d$ and try to show that $\kappa_{\Q}^i$ is given by the minimum degree of an $i$-dimensional face for $i < d$. 

It is possible that these questions may be easier to handle for the case of $k=2$, corresponding to the second row of the Betti table of $R/I_G$. For instance, as pointed in \cite{BabWel}, it's not hard to see that $\kappa_\kk^1(G)$ is less than or equal to the minimum degree of an edge in $\Delta(G)$.  In addition, Abedelfatah and Nevo \cite{AbeNev} have shown that for any coedge ideal $I_G$ the Betti table of $R/I_G$ can never have gaps in the second row, whereas gaps can exist for any larger row.


\subsection{Other models for random ideals}

We note that the Erd\H{o}s--R\'enyi model of random graphs provides another model for random monomial ideals that is related to our study.  For any graph $G$ we can construct an ideal $I_G^*$ whose generators are given by complements of the maximal cliques in $G$. The ideal $I_G^*$ can be recovered as the Stanley--Reisner ideal of the simplicial complex $\nabla(G)$, where $\nabla(G)$ has facets given by $\{[n] \backslash e: e \in G\}$. In particular $I_G^*$ is the Alexander dual of the coedge ideal $I_G$ associated to $G$ (hence the notation). Moreover, for a graph $G$, the simplicial Alexander dual of $\Delta(G)$ is $\nabla(\overline{G})$ where $\overline{G}$ denotes the graph complement of $G$.

In this dual interpretation, if we choose $G \sim G(n,p)$ with $p = n^{-\alpha}$ and $n \rightarrow \infty$ we have a random squarefree monomial ideal where the number of variables $n$ goes to infinity, the degree of the maximal generators as well as Krull dimension of $R/I_G^*$ grow linearly in $n$, but where $R/I_G^*$ has a fixed and finite projective dimension. We also point out that if $1/(d + 1) < \alpha < 1/d$ then almost all generators of $I_{G}^*$ would have degree $n - (d + 1)$ so in some sense these ideals are close to being equigenerated.  Hence as we take $n \rightarrow \infty$ these ideals are similar to the model proposed in  \cite{DHKS} but exhibit quite different homological behavior.

It would be interesting to consider other models of random (hyper)graphs to see what algebraic properties their corresponding Stanley--Reisner rings $R/I_G$ enjoy.  In particular if we were interested in graphs $G$ for which $R/I_G$ has regularity larger than the top dimension of nonvanishing homology of $\Delta(G)$, we would need a model with the property that the clique complex of $G$ has vanishing homology beyond some dimension $d$ but for which induced subcomplexes on some fraction of $n$ vertices has homology in degree larger that $d$. We are not sure if there is a choice of values for the multiparameter model $X(n; p_0 p_1, p_2, p_3, ...)$ that satisfies this property.  We could also consider these questions for other well-studied models for random complexes such as the Vietoris--Rips or \v{C}ech complex of randomly sampled points in a metric space.



\subsection{Normal distributions in the rows of the Betti table}
In \cite{ErmYan}, Erman and Yang consider normal distributions in the asymptotic magnitude of the Betti numbers of $R/I_G$ for $G \sim G(n, c/n)$.  They establish the following for the first row of the Betti table.

\begin{theorem}\cite[Corollary 1.5]{ErmYan}
Fix a constant $0 < c < 1$ and let $\Delta \sim \Delta(n, c/n)$ be a random flag complex. If $i_n$ is an integer sequence converging to $i_n = n/2 + a\sqrt{n}/2$ then 
\[ \frac{\sqrt{2 \pi}}{(1 - c) 2^n \sqrt{n}} \beta_{i_n, i_n + 1}(R/I_G) \rightarrow \exp(-a^2/2).\]
\end{theorem}

In other words, in this first row of the Betti table (and for this choice of $p$) there is a normal distribution among the nonzero Betti numbers among some range of values centered at $i = n/2$. The authors of \cite{ErmYan} conjecture similar behavior along other rows.

\begin{conjecture}\cite[Conjecture 6.4]{ErmYan}
In the case where Theorem 1.3 (cited here as Theorem \ref{thm:density}) yields nonvanishing Betti numbers in row $k$ we conjecture that the $k$th row of the Betti table will be normally distributed in a manner similar to Corollary 1.5. 
\end{conjecture}

Our Theorem \ref{corstaircase} is similar in spirit to this conjecture as it examines the asymptotic distribution of nonzero entries in a single row of the Betti table of $R/I_G$.  Based on this result, we can narrow down where the center of a normal distribution in any row must be (if such a distribution indeed exists).  In particular, for $1/(d + 1) < \alpha < 1/d$ and $k \leq d + 1$, one could look for some type of normal distribution on the magnitude of $\beta_{i, i + k}$ centered at some $i$ between $\Omega(n^{(k-1)\alpha})$ and $O(n^{k \alpha})$, that is in the boxes along the logarithmic diagonal in Figure \ref{summary}. 

We suspect that making precise statements along the lines of Corollary 1.5 of \cite{ErmYan} will involve first having a better understanding of the behavior of $\beta_{d-1}(\Delta; \Q)$ for $\Delta \sim (n, cn^{-1/d})$ where $c$ is a constant. Indeed the assumption that $c$ is between 0 and 1 in Corollary 1.5 of \cite{ErmYan} comes from the \emph{Erd\H{o}s--R\'{e}nyi phase transition} in $G(n, p)$. A classical result in random graph theorem due to Erd\H{o}s and R\'enyi \cite{ErdRen} is that $G(n, p)$ has a one-sided sharp threshold for the existence of cycles at $p = 1/n$. More precisely for $0 < c < 1$, the number of cycles in $G(n, p)$ is Poisson distributed with some bounded mean depending on $c$ while for $c > 1$, $G(n, p)$ will have $\Theta(n)$ cycles with probability tending to 1. This one-sided sharp threshold can naturally be phrased in terms of $\beta_1(\Delta(G); \Q)$ for $G \sim G(n, p)$ and implies that for $0 < c < 1$, $\beta_0(\Delta(G); \Q)$ can be well approximated by the number of vertices minus the number of edges. This is an important part of the proof of Corollary 1.5 of \cite{ErmYan}.

As a step toward settling Erman and Yang's conjecture one could try to prove that for $p = c/\sqrt{n}$ and $c$ a sufficiently small positive constant there is a normal distribution for $\beta_{i_n, i_n + 2}$ for $i_n$ sufficiently close to $n/2$. This in turn would involve finding a sufficiently small constant $c$ so that $\beta_1(\Delta; \Q)$ can be well approximated by $f_1(\Delta) - f_0(\Delta) - f_2(\Delta)$ for $\Delta(n, c/\sqrt{n})$. Based on behavior in other models of random complexes, in particular the Linial--Meshulam model, established by \cite{AroLin}, \cite{LinPel} it seems likely that such a constant exists and Theorem 1.1 of \cite{Kah2014} shows that it is at most $\sqrt{3}$. Recent work of the second author \cite{NewmanSharpThreshold} provides a lower bound on what this critical constant will be if it exists. 

As an alternative way to study the distribution of Betti numbers along a single row of the Betti table, a natural class of ideals to consider would be those which are equigenerated and have a \emph{linear resolution}. Indeed in this case the Betti table of $R/I$ has only a single nonzero row.  For the case of squarefree quadratic ideals this class corresponds to the coedge ideals of \emph{chordal} graphs by a well-known result of Fr\"oberg \cite{Fro}. Among the class of chordal graphs are the class of \emph{threshold} graphs, which admit a natural model of randomness determined by a parameter $p \in (0,1)$. In \cite{EngGoSta} Engstr\"om, Go, and Stamps study the coedge ideal of a random threshold graph and provide a formula for the expected value of each Betti number for any choice of $p$.  A natural extension of this work would be to consider random square-free \emph{stable ideals} generated in degree $d$ (which are $d$-linear), but as far as we know it is an open question to develop an analogous random model for this class of ideals.


\end{document}